\documentclass[10pt,twoside]{siamart1116}

\usepackage{amsmath, fullpage,  url, float, subcaption, rotating, fp,  tikz, enumerate}
\usetikzlibrary{graphs,graphs.standard}\usepackage{dsfont} 
\usepackage{tkz-berge,tkz-graph}
\usepackage{multicol}
\usepackage[mathlines]{lineno}

\usepackage[english]{babel}
\usepackage{graphicx,epstopdf,epsfig}
\usepackage{amsfonts,epsfig,fancyhdr,graphics, hyperref,amsmath,amssymb}
\usepackage{bbm}

\def\cvd{~\vbox{\hrule\hbox{%
     \vrule height1.3ex\hskip0.8ex\vrule}\hrule } }

\renewcommand{\theequation}{\thesection.\arabic{equation}}
\renewtheorem{theorem}{Theorem}[section]

\newtheorem{thm}{Theorem}[section]
\newtheorem{cor}[thm]{Corollary}
\newtheorem{lem}[thm]{Lemma}
\newtheorem{prop}[thm]{Proposition}

\newtheorem{quest}[thm]{Question}
\newtheorem{rem}[thm]{Remark}
\newtheorem{defn}[thm]{Definition}
\newtheorem{ex}[thm]{Example}

\numberwithin{figure}{section}  
\numberwithin{table}{section}   

\usepackage{color}
\definecolor{red}{rgb}{.8,0,0}

\definecolor{blu}{rgb}{0,0,1}

\def\mtx#1{\begin{bmatrix} #1 \end{bmatrix}}
\def\Circ#1#2{\operatorname{Circ}_{#1}\!\left(#2\right)}
\def\ccirc#1#2{\operatorname{Circ}_ {#1}\!\left(\left[#2\right]\right)}
\def\circs#1#2{\operatorname{Circ}_ {#1}\!\left(\{#2\}\right)}

\newcommand{\Z}{{\mathbb Z}}
\newcommand{\R}{{\mathbb R}}

\newcommand{\bx}{{\bf x}}

\newcommand{\by}{{\bf y}}

\newcommand{\D}{\mathcal D}
\newcommand{\PP}{\mathcal P}
\newcommand{\DL}{\D^L}
\newcommand{\spec}{\operatorname{spec}}
\newcommand{\dspec}{\operatorname{spec}_{\D}}
\newcommand{\dLspec}{\operatorname{spec}_{\DL}}
\newcommand{\dev}{\partial}
\newcommand{\dlev}{\partial^L}

\newcommand{\diag}{\operatorname{diag}}
\newcommand{\diam}{\operatorname{diam}}

\newcommand{\tr}{\operatorname{tr}}
\newcommand{\cp}{\, \Box\,}
\newcommand{\OL}{\overline}
\newcommand{\trace}{\operatorname{trace}}

\newcommand{\bit}{\begin{itemize}}
\newcommand{\eit}{\end{itemize}}
\newcommand{\ben}{\begin{enumerate}}
\newcommand{\een}{\end{enumerate}}
\newcommand{\beq}{\begin{equation}}
\newcommand{\eeq}{\end{equation}}
\newcommand{\bea}{\begin{eqnarray*}}
\newcommand{\eea}{\end{eqnarray*}}
\newcommand{\bean}{\begin{eqnarray}}
\newcommand{\eean}{\end{eqnarray}}
\newcommand{\bpf}{\begin{proof}}
\newcommand{\epf}{\end{proof}}
\newcommand{\x}{\times}

\newcommand{\lf}{\left\lfloor}
\newcommand{\rf}{\right\rfloor}

\newcommand{\LL}{\lambda}

\begin{document}

\bibliographystyle{plain}

\setcounter{page}{1}

\thispagestyle{empty}

 
\title{Graphs that are cospectral for the distance Laplacian}
\author{Boris Brimkov\thanks{Department of Computational and Applied Mathematics, Rice University, Houston, TX, 77005, USA (boris.brimkov@rice.edu)}
\and 
Ken Duna\thanks{Department of Mathematics, University of Kansas, Lawrence, KS, 66045, USA (kduna@ku.edu)} 
\and 
Leslie Hogben\thanks{Department of Mathematics, Iowa State University, Ames, IA 50011, USA and 
American Institute of Mathematics, 600 E. Brokaw Road, San Jose, CA 95112, USA (hogben@aimath.org).}
\and 
Kate Lorenzen\thanks{Department of Mathematics, Iowa State University, Ames, IA 50011, USA (lorenkj@iastate.edu, reinh196@iastate.edu, sysong@iastate.edu) } 
\and \hspace{200pt}
Carolyn Reinhart\footnotemark[4] 
\and 
Sung-Yell Song\footnotemark[4]
\and Mark Yarrow\thanks{School of Mathematics and Statistics, University of Sheffield, Sheffield, S3 7RH, United Kingdom (MAYarrow1@sheffield.ac.uk)}
}


\maketitle


\begin{abstract} 
The distance matrix $\D(G)$ of a graph $G$ is the matrix containing the pairwise distances between vertices, and the distance Laplacian matrix is $\DL(G)=T(G)-\D(G)$, where $T(G)$ is the diagonal matrix of row sums of $\D(G)$. We establish several general methods for producing  $\DL$-cospectral graphs that can be used to construct infinite families.  We provide examples showing that various properties are not preserved by $\DL$-cospectrality, including examples of $\DL$-cospectral strongly regular and circulant graphs.  We establish that the absolute values of coefficients of the distance Laplacian characteristic polynomial are decreasing, i.e., $|\delta^L_{1}|\geq \dots \geq |\delta^L_{n}|$ where $\delta^L_{k}$ is the coefficient of $x^k$. \end{abstract}


\section{Introduction}\label{s:intro}$\null$

Spectral graph theory is the study of graphs via the eigenvalues of a matrix defined from the graph. Often  information about the graph  can be recovered from the spectrum of an associated matrix $M(G)$, but the existence of $M$-cospectral graphs  implies there is information that cannot be recovered from the spectrum. Nonisomorphic graphs $G$ and $H$ are said to be \emph{$M$-cospectral}  if $\spec(M(G))=\spec(M(H))$.    Constructions that switch certain edges to produce adjacency cospectral graphs and distance cospectral graphs  have been introduced by   Godsil and McKay   \cite{GM82} and Heysse  \cite{H17}. Aouchiche and Hansen give a thorough discussion of the history of $M$-cospectrality for various types of matrices $M$ defined from a graph in \cite[Section 2]{AH13}. 

This paper focuses on the distance Laplacian matrix of a graph,  which was introduced in \cite{AH13} and  is  defined from  the distance matrix in analogy with the way  the Laplacian matrix is defined from the adjacency matrix  (precise definitions of graph theory terms are provided below).
 The {\em distance matrix} $\D(G)$ 
 of a connected graph $G$ is the matrix indexed by the vertices $\{v_1,\dots,v_n\}$  of $G$ having $i,j$-entry
 equal to the {distance} between the vertices $v_i$ and $v_j$. 
For $v\in V(G)$, the {\em transmission index} of $v$ is $t(v)=\sum_{w\in V(G)} d(v,w)$.    The \emph{distance Laplacian matrix} is 
 $\DL(G)=T(G)-\D(G)$ where $T(G)=\diag(t(v_1),\dots,t(v_n))$. Distance matrices were introduced  \cite{GP71} for analysis of the loop switching problem in telecommunications, which involves finding appropriate addresses so that a message can move efficiently through a series of loops from its origin to its destination, choosing the best route at each switching point.  Recently there has been renewed interest in the loop switching problem and extensive work on distance spectra; see \cite{AH14} for a  survey.  
 The spectrum of $\DL(G)$ has been related to various graph parameters such as the  connectivity of the complement of $G$ \cite{AH13}.

In Section \ref{ss:preserve} we provide examples to show that various graph properties are not preserved by $\DL$-cospectrality. We establish some general methods for producing  $\DL$-cospectral graphs and apply one to construct an infinite family of $\DL$-cospectral bipartite pairs in Section \ref{s:construct}.  
These  constructions do not generally produce cospectral distance matrices because they rely on the fact that $\DL(G)$ is positive semidefinite with all row sums equal to zero.  
It is clear that if $T(G)$ is a scalar matrix, then there are formulas that determine the eigenvalues of $\DL(G)$ from the eigenvalues of $\D(G)$ (and vice versa). Thus, the study of the distance spectrum and the study of the distance Laplacian spectrum  are the same for transmission regular graphs; we apply this to  $\DL$-cospectral strongly regular and circulant graphs in Section  \ref{s:TR}.  In Section \ref{s:unimod} we establish that the absolute values of coefficients of the distance Laplacian characteristic polynomial are log-concave, unimodal, and in fact decreasing, i.e., $|\delta^L_{1}|\geq \dots \geq |\delta^L_{n}|$ where $|\delta^L_{k}|$ is the coefficient of $x^k$.  The remainder of this introduction contains additional definitions and terminology.

All graphs in this paper are simple (no loops or multiedges) and undirected.  For a graph $G$, the set of vertices is denoted by $V(G)$ and the set of edges (2-element subsets of vertices) is denoted by $E(G)$; an edge $\{u,v\}$ is also denoted by $uv$.  
The order of $G$ is the number $n=|V(G)|$ of vertices. The edge $e=uv$ is {\em incident} to vertices $u$ and $v$, and $u$ and $v$ are said to be {\em adjacent}. Adjacent vertices are called  {\em neighbors}, and the {\em neighborhood} of vertex $v$, $N(v)$, is the set of all neighbors of $v$. 
The {\em degree} of a vertex $v$ is $\deg v=|N(v)|$. 
If  there exists a $k$ such that $k=\deg(v)$ for all $v\in V(G)$, then $G$ is said to be {\em $k$-regular} or {\em regular}. 
 A {\em $u,v$-path} in $G$ is a list of vertices that start at $u$ and end at $v$ such that any two consecutive vertices in the list form an edge in $E(G)$ and no vertex is repeated. 
A graph is {\em connected} if for every pair of vertices $u,v$ there exists a $u,v$-path. The length of a path is one less than the number of vertices (i.e., is the number of edges), and the {\em distance} between two vertices $d(u,v)$ is the length of the shortest $u,v$-path.
A graph must be connected in order for the distance or distance Laplacian matrix to be defined. 
An {\em isomorphism} from graph $G$ to graph $H$ is a bijection $f:V(G) \to V(H)$ such that  $uv \in E(G)$ if and only if $f(u)f(v) \in E(H)$ for all $u,v\in V(G)$.

A {\em subgraph} of $G$ is a graph $H$ such that $V(H) \subseteq V(G)$ and $E(H) \subseteq E(G)$. If $G$ is a graph, $u,v\in V(G)$, and $uv\not\in E(G)$, then $G+uv$ is the graph obtained from $G$ by adding edge $uv$, i.e., $V(G+uv)=V(G)$ and $E(G+uv)=E(G)\cup\{uv\}$; $G$ is a subgraph of $G+uv$.  Given $S\subseteq V(G)$, the {\em induced subgraph} $G[S]$ is the subgraph of $G$ with vertex set $S$ and edge set consisting of all edges of $G$ that have both endpoints in $S$. 
A graph $G$ is {\em bipartite} if $V(G)=A \cup B$, $A \cap B = \emptyset$, and for every $e \in E$, exactly one vertex of $e$ is in $A$ and exactly one vertex is in $B$; in this case $A$ and $B$ are called 
{\em parts}. A graph is {\em planar} if it can be drawn in the plane without any edges crossing. 
The {\em diameter} of $G$, denoted by $\diam(G)$, is the greatest distance between any two vertices in $G$. 
A {\em cycle} is a list of vertices in  $G$ such that any two consecutive vertices in the list form an edge in $E(G)$ and no vertex is repeated except that the last vertex equals the first vertex.
The {\em girth} of  $G$ is the length of the shortest cycle in $G$. 
 A {\em dominating vertex} is a vertex adjacent to every other vertex.
A {\em cut-vertex} in a connected graph is a vertex whose deletion disconnects the graph (to {\em delete} a vertex from a graph means to delete the vertex and all edges incident with it).

The {\em adjacency  matrix} of a graph $G$ is the matrix indexed by the vertices $\{v_1,\dots,v_n\}$  of $G$ having $i,j$-entry  equal to one if $\{i,j\}$ is an edge and zero otherwise, and is denoted by $A(G)$.   
The \emph{Laplacian matrix} of $G$ is  $L(G)=\diag(\deg v_1,\dots,\deg v_n)-A(G)$.
 The eigenvalues of $\D(G)$ are called {\em distance eigenvalues} and are denoted by $\dev_1\ge  \dev_2\ge \dots\ge \dev_n$. The {\em distance spectrum} of $G$ is the multiset of distance eigenvalues of $G$ and is denoted by   $\dspec(G)$. 
 The {\em distance Laplacian eigenvalues} $\dlev_1=0\le  \dlev_2\le \dots\le \dlev_n=\rho(\DL(G)) $ are eigenvalues of $\DL(G)$, where $\rho$ denotes the spectral radius (observe that the ordering of the eigenvalues is reversed).  The {\em distance Laplacian spectrum} of $G$ is the multiset of distance Laplacian eigenvalues of $G$ and is denoted by   $\dLspec(G)$.
 For a graph $G$ of order $n$, the {\em average transmission} of $G$ is $t(G)=\frac 1 n \sum_{i=1}^n t(v_i)$. 
 We say  $G$  is {\em transmission regular} if $t(v)=t(u)$ for all $v,u\in V(G)$.  If $G$ is a transmission regular graph, then the {\em transmission} of $G$ is the average transmission of $G$, or equivalently $t(G)=t(v)$ for any vertex $v$ of $G$.



 \section{Preservation of parameters by $\DL$-cospectrality}\label{ss:preserve}
 A graph parameter $\zeta$ is {\em preserved by $\DL$-cospectrality} if $\dLspec(G)=\dLspec(H)$ implies $\zeta(G)=\zeta(H)$; note that we allow true-false as possible values of $\zeta$ in addition to numerical values.  In this section we list some  parameters that are  preserved by $\DL$-cospectrality and give examples showing others are not preserved.

 The order of a graph is preserved because it is the number of eigenvalues.
 The trace of $\DL(G)$ is preserved  because the trace of a matrix is the sum of its eigenvalues.  Therefore the {\em Wiener index}  $W(G):=\frac 1 2 \sum_{i=1}^n \sum_{j=1}^n d(v_i,v_j)=\frac 1 2 \trace(\DL(G))$ 
 and the average transmission are preserved, since $t(G)=$ $\frac 1 n \sum_{i=1}^n t(v_i)=\frac 1 n \trace(\DL(G))$.
Furthermore, Aouchiche and Hansen have shown that the number of components of the complement $\OL G$ is  preserved by $\DL$-cospectrality  \cite{AH13}.

A computer search for small $\DL$-cospectral graphs is effective for finding examples that show certain parameters are not preserved by $\DL$-cospectrality, and also for finding examples of $\DL$-cospectral graphs that illustrate a family. A list of  $\DL$-cospectral graphs of order at most ten in {\em graph6} format is presented in \cite{sagedata} and the {\em Sage} code used to find these graphs is given in \cite{sagecode}.
Verification of  the $\DL$-cospectrality of the examples that follow is given in \cite{sageverify}.

Let  $G$ be a graph of order  $n$.  If the degrees of the vertices of $G$ are $d_1\le \dots \le d_n$, then the $n$-tuple $(d_1, \dots, d_n)$ is the {\em degree sequence} of $G$.
 If the transmission indices of the vertices of $G$ are $t_1\le \dots \le t_n$, then the $n$-tuple $(t_1, \dots, t_n)$ is the {\em transmission sequence} of $G$.

\begin{figure}[h!]
\begin{center}
\scalebox{.3}{\includegraphics{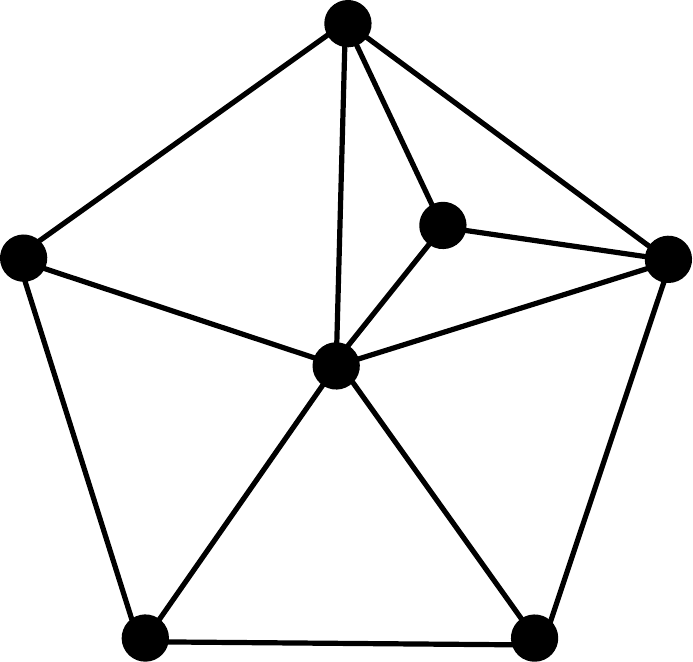}}\qquad\qquad\scalebox{.3}{\includegraphics{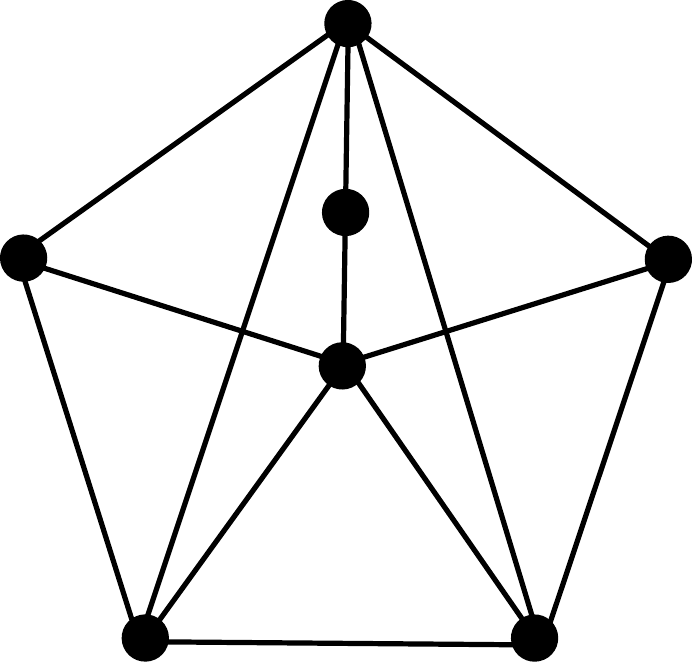}}\\
$\null$\  \ $G_1$\qquad\qquad\qquad\qquad\ \ \ $G_2$
\caption{$\DL$-cospectral graphs that have different degree and transmission sequences (see Example \ref{ex:deg-seq}).  \label{fig:deg-seq}\vspace{-10pt}}
\end{center}
\end{figure}

\begin{ex}\label{ex:deg-seq}{\rm 
 The graphs $G_1$ and $G_2$  shown in Figure \ref{fig:deg-seq} are $\DL$-cospectral because $\dLspec(G_1)=\dLspec(G_2)=\left\{0,7,10-\sqrt{2},9,10, 10+\sqrt{2},12\right\}$. Since the degree sequences   are $(3, 3, 3, 3, 4, 4, 6)$ 
  for  $G_1$ and  $(2, 3, 3, 4, 4, 5, 5)$ for  $G_2$,   the degree sequence is not preserved by $\DL$-cospectrality. Since the transmission sequences   are $(6, 8, 8, 9, 9, 9, 9)$ for  $G_1$ and  $(7, 7, 8, 8, 9, 9, 10)$ for  $G_2$,   the transmission sequence is not preserved by $\DL$-cospectrality. }\end{ex}

 \begin{figure}[h!]
\begin{center}
\scalebox{.3}{\includegraphics{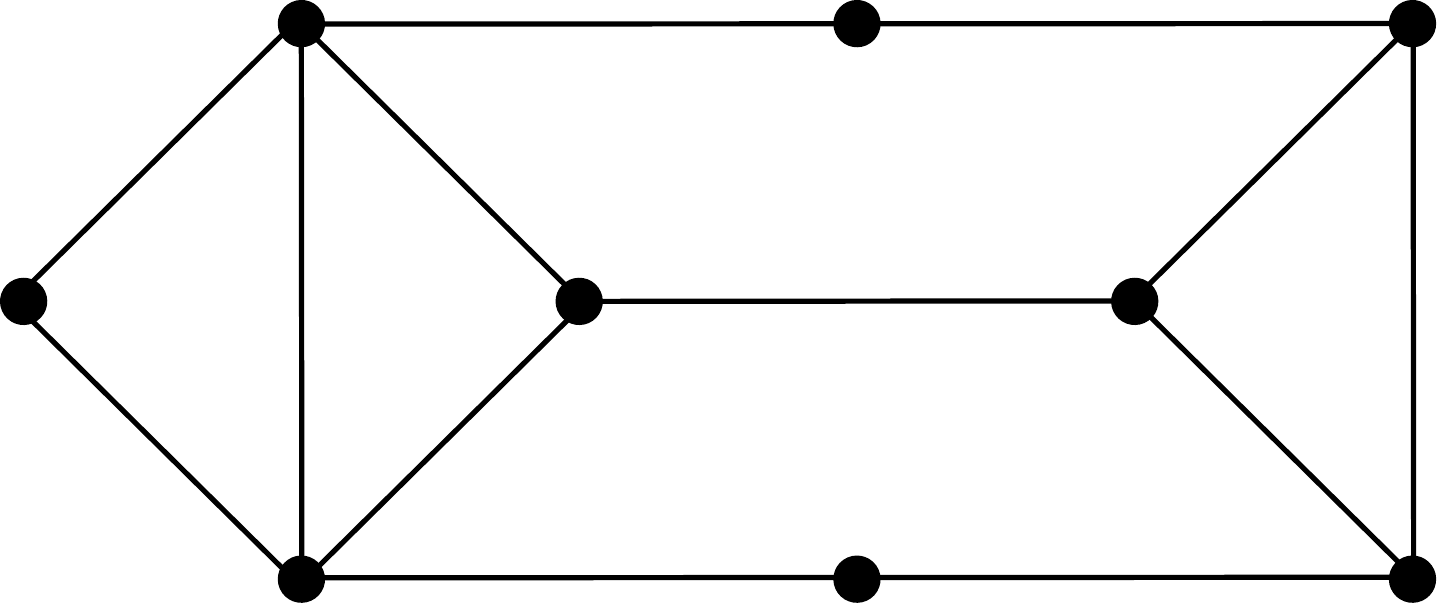}}\qquad\scalebox{.3}{\includegraphics{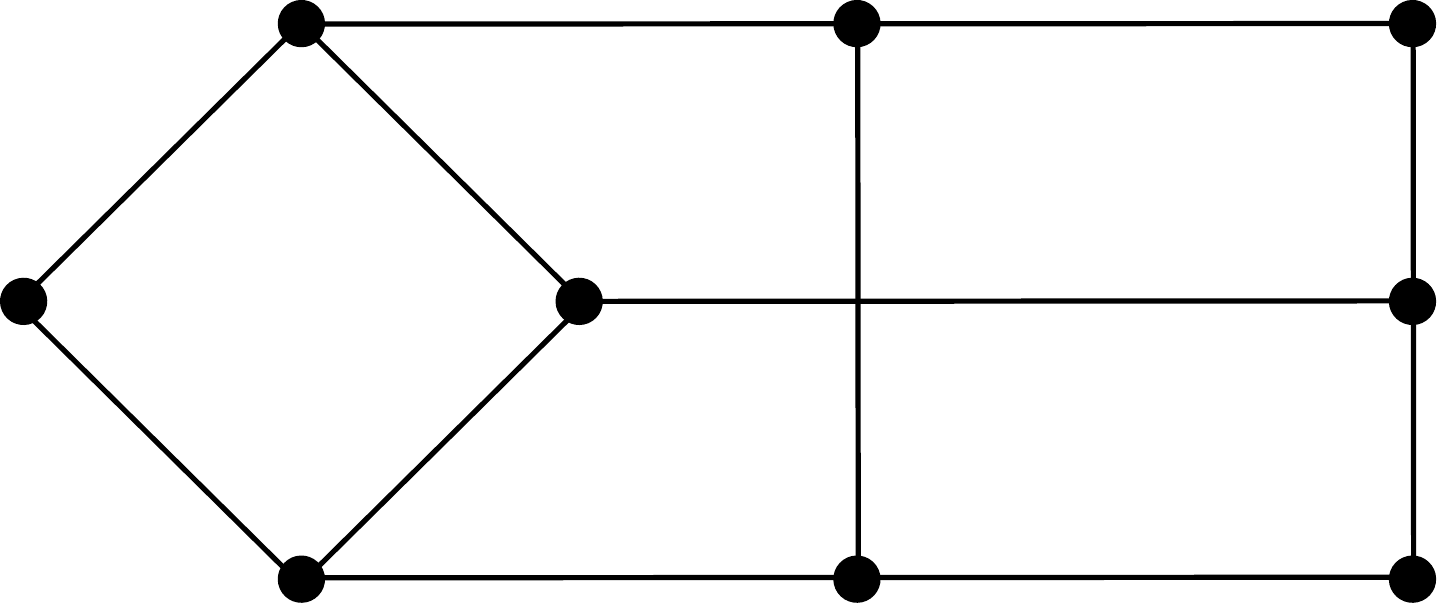}}\\
$\null$\qquad \ $G_1$\qquad\qquad\qquad\qquad\qquad\qquad$G_2$
\caption{$\DL$-cospectral graphs that have different numbers of edges, girths, and multisets of distances (see Example \ref{ex:edge-girth}).  \label{fig:edge-girth}\vspace{-16pt}}
\end{center}
\end{figure}

\begin{ex}\label{ex:edge-girth}{\rm 
 The graphs $G_1$ and $G_2$  shown in Figure \ref{fig:edge-girth} are $\DL$-cospectral because the characteristic polynomials are the same: $p_{\DL(G_1)}(x)=p_{\DL(G_2)}(x)=x^9 - 130x^8 + 7369x^7 - 237912x^6 + 4785415x^5 - 61411718x^4 +
491069091x^3 - 2237203064x^2 + 4446151812x$. Observe that   $G_1$ has 13 edges, girth 3, and 5 pairs of vertices at distance 3 from each other, whereas  $G_2$ has 12 edges, girth 4,  and 6 pairs of vertices at distance 3.  Thus, the number of edges, the girth, and the multiset of distances are not  preserved by $\DL$-cospectrality.  }\end{ex}

Figure \ref{fig:diam} presents a pair of  $\DL$-cospectral graphs showing that  $\DL$-cospectrality does not preserve the diameter. Figures  \ref{fig:leaf}, \ref{fig:dom}, \ref{fig:cut}, and \ref{fig:aut} present pairs of  $\DL$-cospectral graphs that establish  $\DL$-cospectrality does not preserve the following properties:   having a leaf, having a dominating vertex, having a cut-vertex, and having a nontrivial automorphism.

\begin{figure}[h!]
\begin{center}
\includegraphics[scale=0.3]{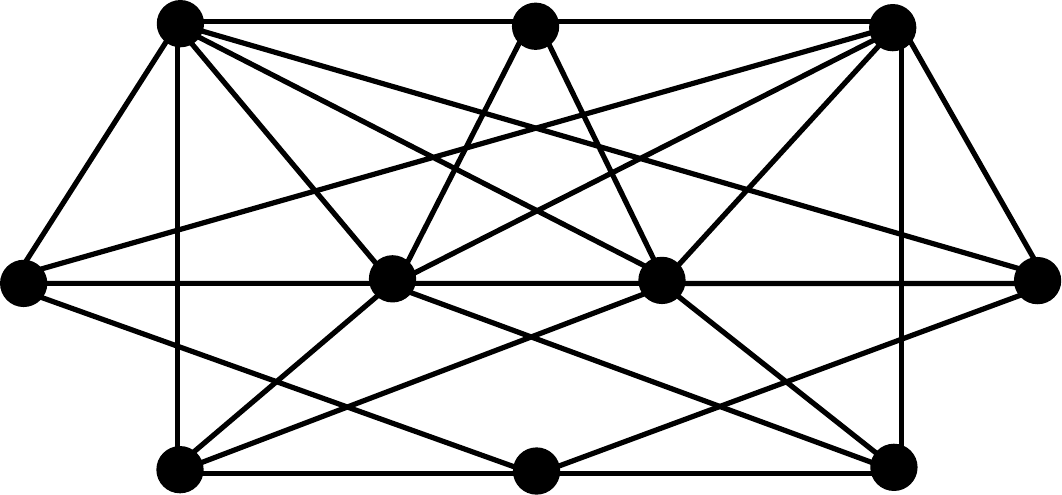}\qquad\quad
\includegraphics[scale=0.3]{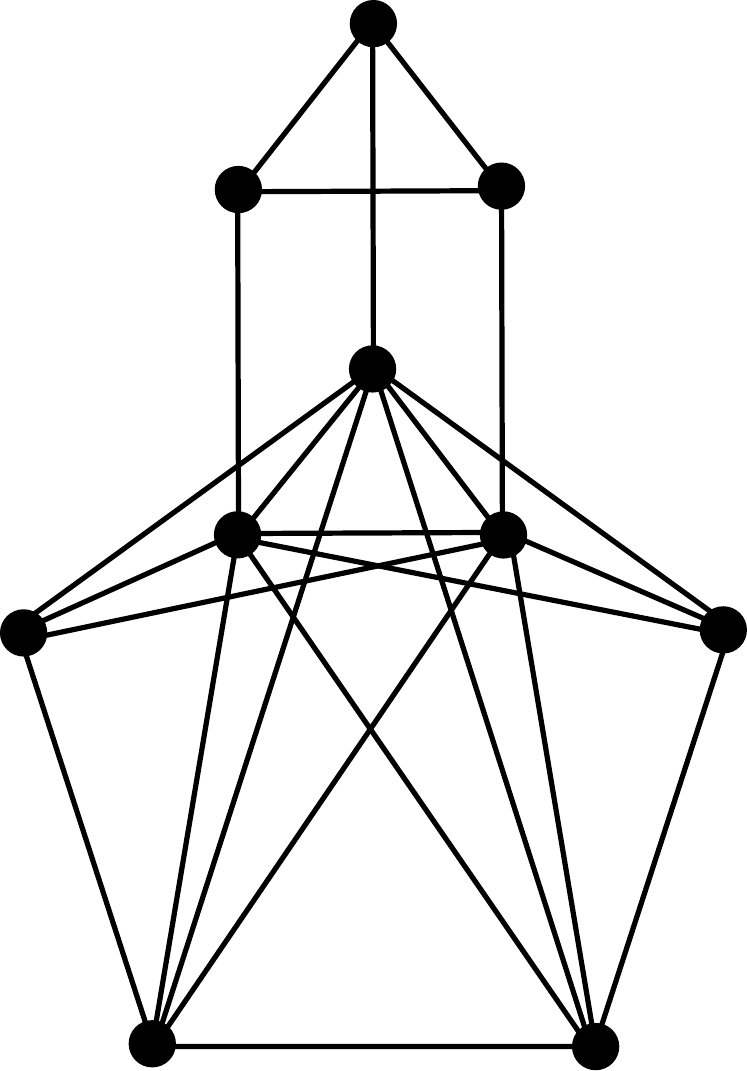}\\
$\null$ \ \ $G_1$\qquad\quad\qquad\quad\qquad\quad\ \ $G_2$
\caption{Observe that $\diam(G_1)=3$, $\diam(G_2)=2$, and $p_{\DL(G_1)}(x)=p_{\DL(G_2)}(x)=x^{10} - 132x^9 + 7720x^8 - 262558x^7 + 5722578x^6 - 82891102x^5 +
797942816x^4 - 4922528058x^3 + 17658758885x^2 - 28066657350x$. \label{fig:diam}}\vspace{-15pt}
\end{center}
\end{figure}

\begin{figure}[h!]\begin{center}
\includegraphics[scale=0.3]{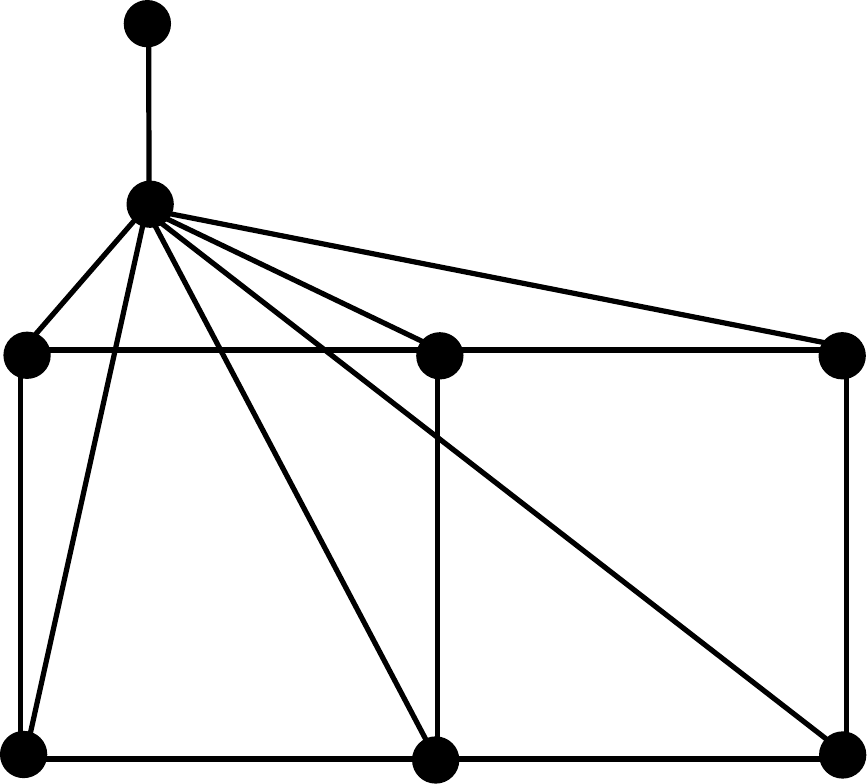}\qquad\quad\includegraphics[scale=0.3]{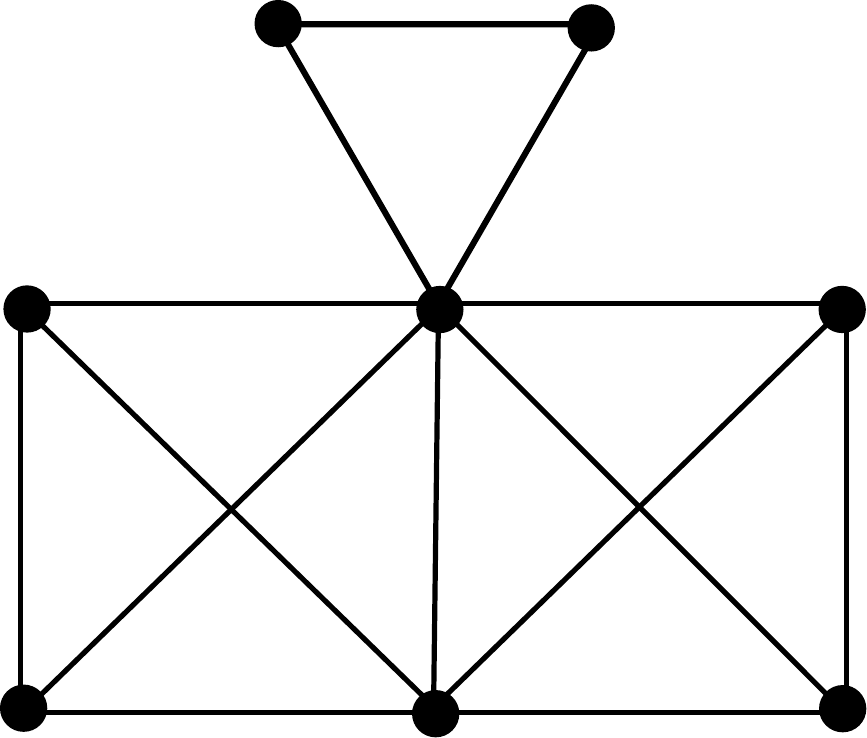}\\
$G_1$\qquad\quad\qquad\quad\qquad\quad$G_2$\\
\caption{   Observe that $\dLspec(G_1)=\dLspec(G_2)=\{0, 8, 10, 12, 12,  13, 14, 15  \}$ and $G_1$ has a leaf whereas $G_2$  does not. \label{fig:leaf}}\vspace{-15pt}
\end{center}\end{figure}

\begin{figure}[h!]\begin{center}
\includegraphics[scale=0.3]{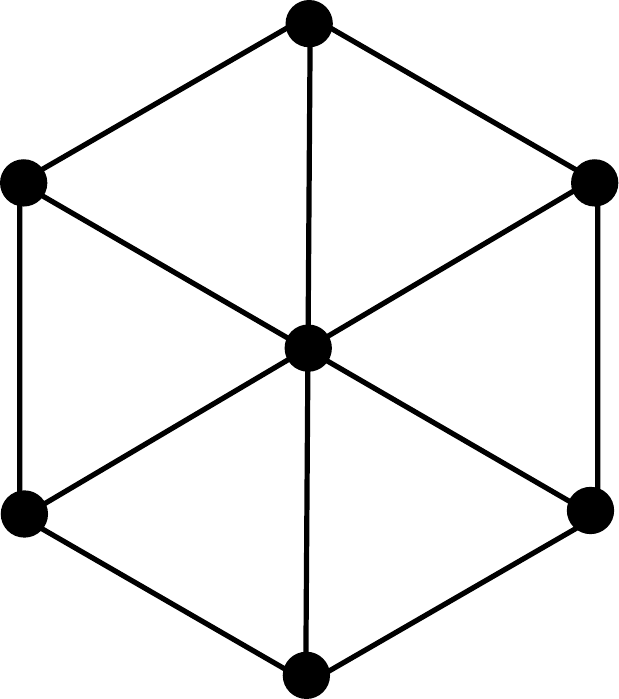}\qquad\quad\includegraphics[scale=0.3]{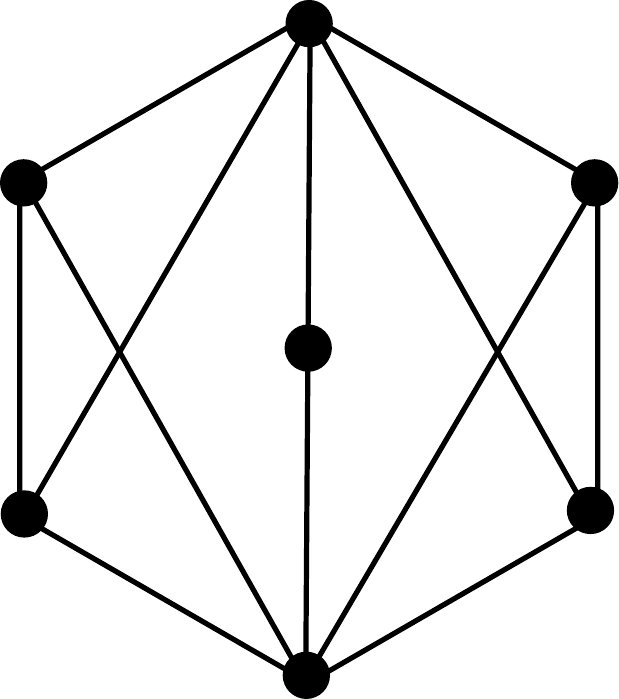}\\
$G_1$\qquad\quad\qquad\quad\qquad  $G_2$\\
\caption{     Observe that $\dLspec(G_1)=\dLspec(G_2)=\{ 0,7,9, 10, 10, 12, 12\}
$ and $G_1$ has a dominating vertex whereas $G_2$ does not. \label{fig:dom}}
\end{center}\end{figure}

\begin{figure}[h!]\begin{center}
\includegraphics[scale=0.3]{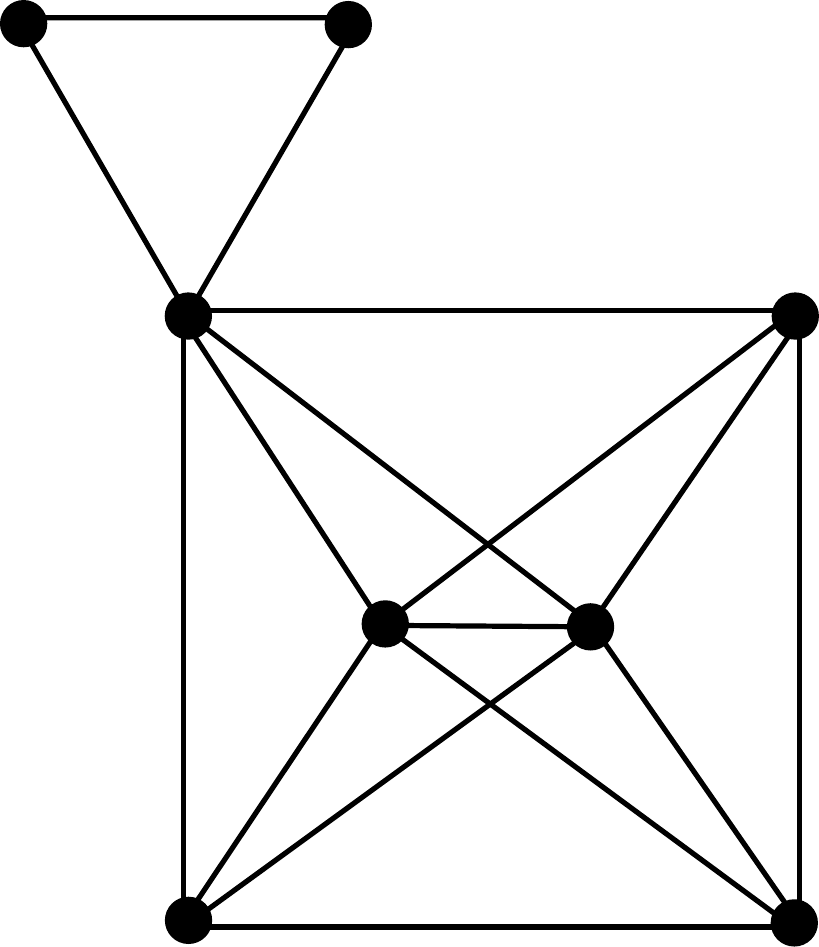}\qquad\quad\includegraphics[scale=0.3]{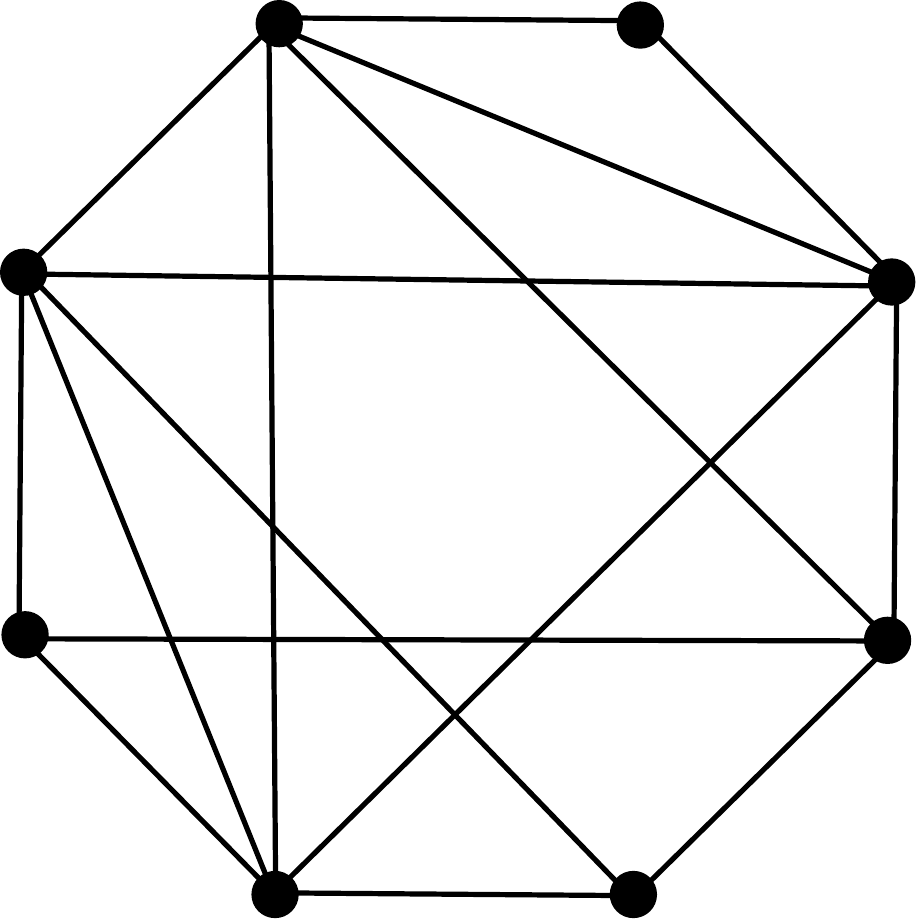}\\
$G_1$\qquad\quad\qquad\quad\qquad\quad$G_2$\\
\caption{ Observe that $p_{\DL(G_1)}(x)=p_{\DL(G_2)}(x)=x^8 - 84x^7 + 3000x^6 - 59072x^5 + 692816x^4 - 4841152x^3 +
18666240x^2 - 30643200x$ and $G_1$ has a cut vertex whereas $G_2$ does not. \label{fig:cut}}\vspace{-15pt}
\end{center}\end{figure}

\begin{figure}[h!]
\begin{center}
\includegraphics[scale=0.3]{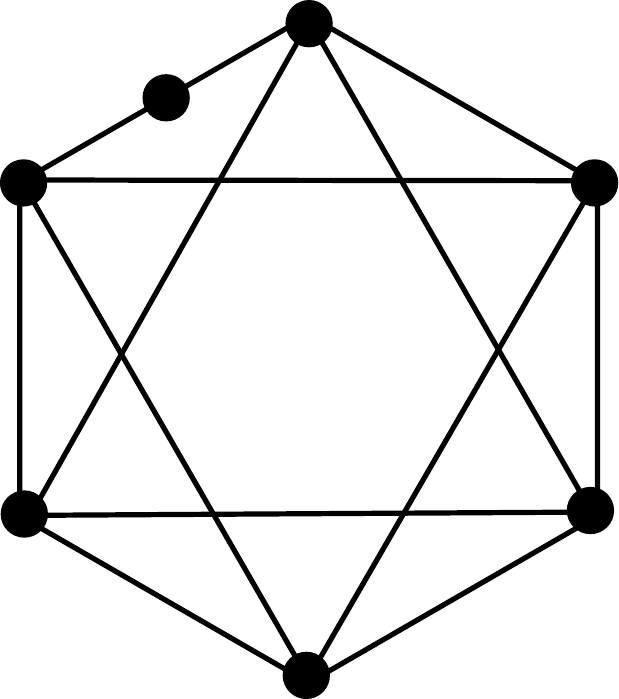}\qquad\quad
\includegraphics[scale=0.3]{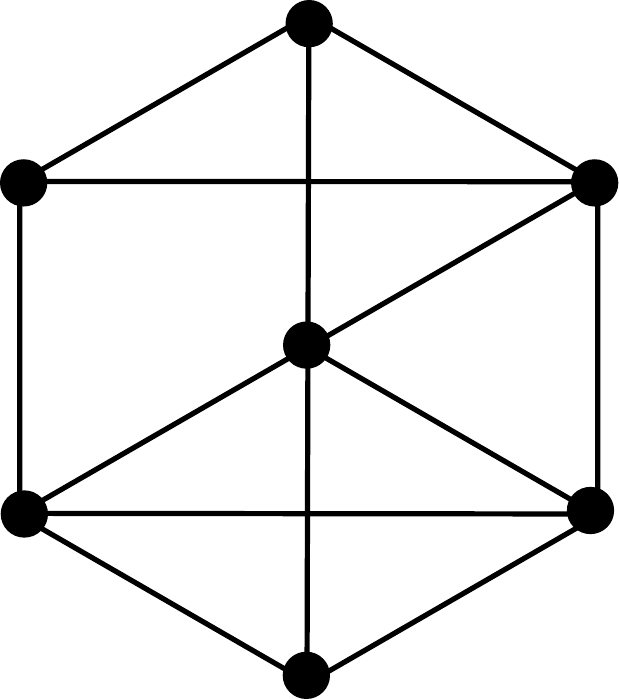}\\
$G_1$\qquad\quad\qquad\quad\qquad $G_2$
\caption{  Observe that $p_{\DL(G_1)}(x)=p_{\DL(G_2)}(x)=x^8 - 84x^7 + 3007x^6 - 59448x^5 + 700756x^4 - 4923264x^3 +
19080000x^2 - 31449600x$ and $G_1$ has symmetry, i.e., it has a nontrivial automorphism, whereas $G_2$ does not, i.e., it only has the trivial automorphism. \label{fig:aut}}\vspace{-15pt}
\end{center}
\end{figure}

A {\em clique} is a set $S\subseteq V(G)$ such that $x,y\in S$ implies $x$ and $y$ are adjacent, and an {\em independent set} is a set $S\subseteq V(G)$ such that $x,y\in S$ implies $x$ and $y$ are not adjacent.  The {\em clique number} $\omega(G)$ is the maximum cardinality of a clique of $G$, and {\em independence number} $\alpha(G)$ is the maximum cardinality of an independent  set of $G$.
Example \ref{ex:S-H24} below shows that the clique number and independence number  are not preserved by $\DL$-cospectrality.
Example \ref{Paley29} shows that the property of being a  circulant graph is not preserved (see Section \ref{sscirc} for the definition of a circulant graph). 
The bipartite cospectral pair on eight vertices shown in Figure \ref{fig:bipartite8} establishes that planarity is not preserved, because $B_{1}+v_1v_3$ is  planar but $B_{1}+v_2v_4$ is not planar. Their $\DL$-cospectrality is proved  in Theorem \ref{thm:bipart} below. 

\begin{figure}[h!]
\begin{center}
\includegraphics[scale=0.3]{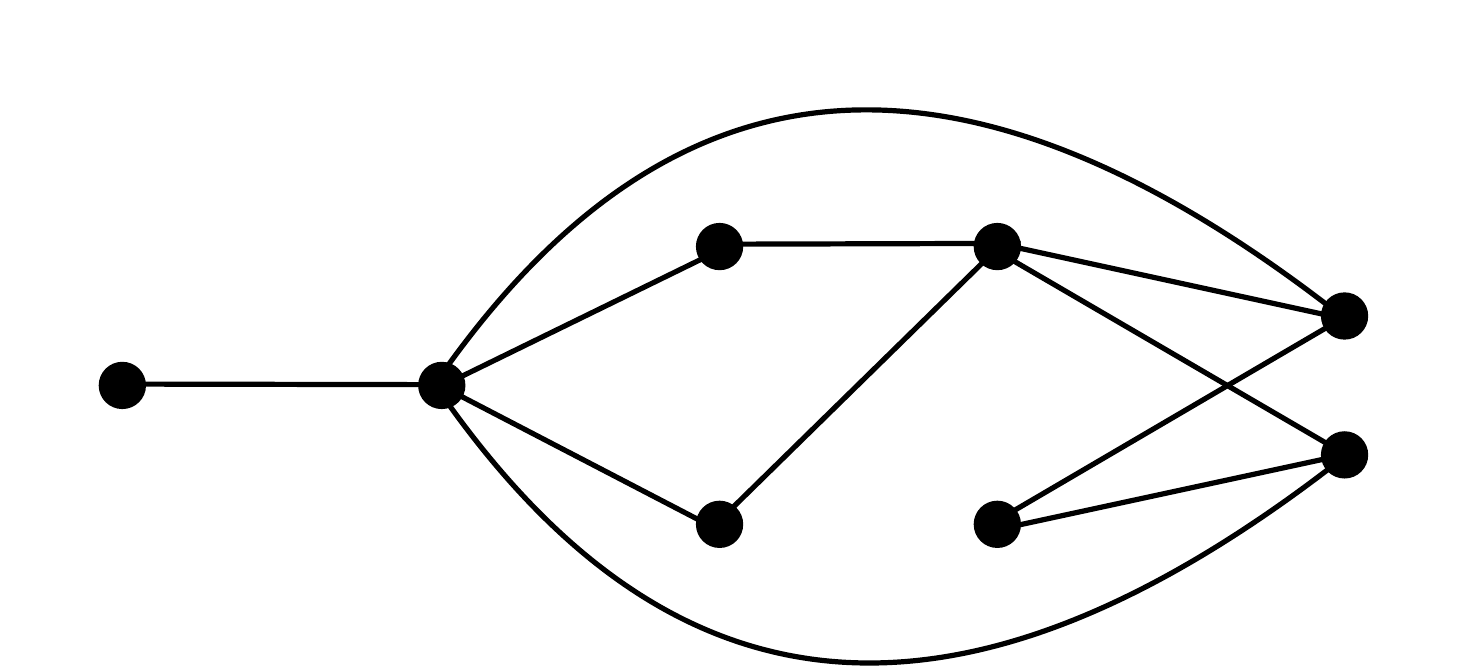}\qquad\quad
\includegraphics[scale=0.3]{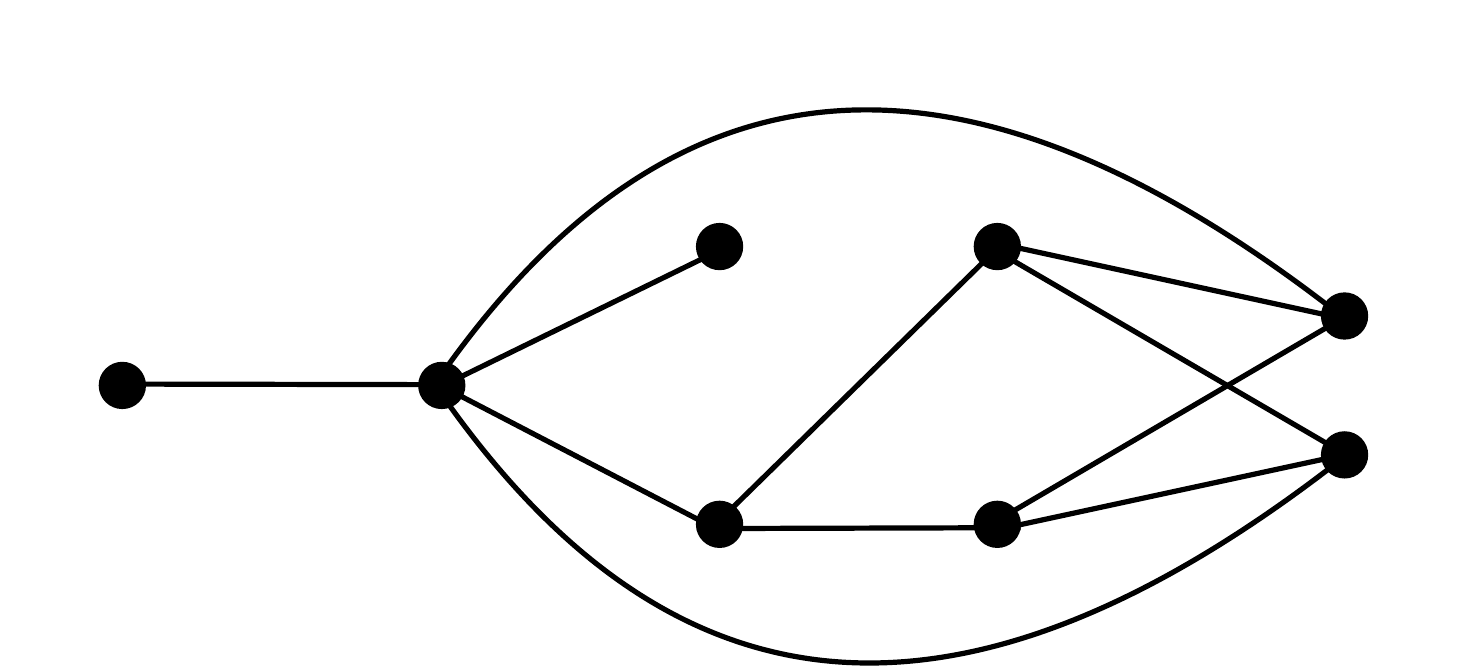}\\
$B_{1}+v_1v_3$\quad\qquad\qquad\quad\qquad\quad\qquad\quad$B_{1}+v_2v_4$
\caption{Observe that $p_{\DL(B_{1}+v_1v_3)}(x)=p_{\DL(B_{1}+v_2v_4)}(x)=x^8 - 98x^7 + 4087x^6 - 94020x^5 + 1288463x^4 - 10517842x^3 +
47349497x^2 - 90671880x
$ and $B_{1}+v_1v_3$  is planar whereas $B_{1}+v_2v_4$ is not. \label{fig:bipartite8}}\vspace{-10pt}
\end{center}
\end{figure}

\section{Constructions of $\DL$-cospectral graphs}\label{s:construct}

In this section we provide constructions of $\DL$-cospec-tral graphs that use sets of twin vertices, or sets of vertex pairs with a relaxation of the  twin structure. One of these constructions explains the only pair of cospectral bipartite graphs on eight vertices and provides a way to build large  bipartite $\DL$-cospectral graphs on an even number of vertices. 


Let $v_1, v_2$ be vertices of graph $G$ such that $N(v_1)=N(v_2)$. Such vertices are called  {\em  twins} (technically, these are {\em independent} or {\em non-adjacent} twins, but we use the term {\em twins} to mean {\em independent twins}).  Note that twins have the same transmission.  

\begin{defn}{\rm 
If $v_1$ and $v_2$ are twins, $v_3$ and $v_4$ are twins, and $t(v_1)=t(v_3)$, then we say that  $\{\{v_1, v_2\},\{ v_3, v_4\}\}$ is a set of {\em co-transmission twins.} }
\end{defn}

In \cite{NP}, the authors prove that if $v_1$ and $v_2$ are  twin vertices in graph $G$ where $\dLspec(G)$ and $\dLspec(G+v_1v_2)$ differ only by one eigenvalue of $G$ decreasing by $2$ while the others remain the same, then the changed eigenvalue of $G$ is $t(v_1)+2$. 
We establish a stronger statement that we use to prove our first construction.

 \begin{prop}\label{lem:twins} Suppose $v_1$ and $v_2$ are twins.  Then there is an orthogonal basis of eigenvectors of $\DL(G)$ that includes $[1 , -1 ,0 , \dots , 0]^T$, which is associated with  eigenvalue $t_G(v_1)+2$, and all other eigenvectors in the basis  are of the form $[a,a,y_3,\dots ,y_n]^T$ for some $a, y_i\in\R$. Furthermore, this basis of eigenvectors for $\DL(G)$ is also a  basis of eigenvectors for $\DL(G+v_1v_2)$, the eigenvalue of $\DL(G+v_1v_2)$ for eigenvector $[1 , -1 ,0 , \dots , 0]^T$  is $t_{G}(v_1)$, and the eigenvalues of $\DL(G+v_1v_2)$ for other eigenvectors are unchanged. 
\end{prop}

\begin{proof}
Since $v_1$ and $v_2$ are twins, it follows that $(\DL(G))_{j1}=(\DL(G))_{j2}$ for all $j=3, \ldots, n$ and $t(v_1)=t(v_2)$. 
In addition, $v_1v_2 \not \in E(G)$ and there exists a path of length two between $v_1$ and $v_2$ because they have the same neighborhood. 
Therefore the $2 \times 2$ principal submatrix associated with $v_1$ and $v_2$ is of the form $\mtx{
t_G(v_1) & -2 \\
-2 & t_G(v_2)
}$. 
For $\bx=[1,-1 ,0,\dots,0]^T$, $\DL(G)\bx=[t_G(v_1)+2,-t_G(v_2)-2,0,\dots,0]^T$ because after the first two rows, all the  entries of the first two columns are equal. Since $t_G(v_1)=t_G(v_2)$, it follows that $\bx$ is an eigenvector with eigenvalue $t_G(v_1)+2$. 
Since $\DL(G)$ is a real symmetric matrix, it has a basis of  orthogonal eigenvectors.  Thus there is a basis of eigenvectors in which every eigenvector $\by\ne\bx$ is  of the form $\by=[a,a,y_3,\dots ,y_n]^T$ for some $a\in \mathbb{R}$. 

Since $v_1,v_2$ share a neighborhood, it follows that adding the edge $v_1v_2$ to $G$ changes only the distance between $v_1$ and $v_2$. Therefore \vspace{-10pt}
\[\DL(G+v_1v_2)=\DL(G)+\left(\mtx{-1 & 1 \\1 & -1}\oplus O_{n-2,n-2}\right).\vspace{-10pt}\] 
 For $\bx=[1, -1, 0 \dots, 0]^T$\!,\vspace{-10pt}
  \[\DL(G+v_1v_2)\bx=[t_G(v_1)+2, -t_G(v_1)-2, 0, \dots, 0]^T+[-2, 2, 0, \dots, 0]^T=[t_G(v_1), -t_G(v_1), 0, \dots, 0]^T\!.\vspace{-10pt}\] 
 Thus $\bx$ is an eigenvector of $G+v_1v_2$ with eigenvalue $t_G(v_1)$. 
Let $\by=[a,a,y_3,\dots ,y_n]^T$ be an eigenvector of $\DL(G)$. Since $\DL(G+v_1v_2)=\DL(G)+\left(\mtx{-1 & 1 \\1 & -1}\oplus O_{n-2,n-2}\right)$,   $\DL(G+v_1v_2) \by =\DL(G) \by$. 
Therefore the eigenvalues of $\DL(G+v_1v_2)$ associated with other eigenvectors in the orthogonal basis of eigenvectors of $\DL(G)$ are unchanged. 
\end{proof}
 

\begin{cor} \label{cor:cospecAddEdge}
Let $G$ and $H$ be $\DL$-cospectral, and let $v_1, v_2 \in V(G)$, $v_3, v_4 \in V(H)$ be such that $v_1, v_2$ are  twins, $v_3,v_4$ are  twins and $t(v_1)=t(v_3)$. Then $G+v_1v_2$ and $H+v_3v_4$ are $\DL$-cospectral.   In particular, if $G$ is a graph with a set of  co-transmission twins $\{\{v_1, v_2\},\{ v_3, v_4\}\}$, then $G+ v_1v_2$ and $G+v_3v_4$ are $\DL$-cospectral. 
\end{cor}

 \begin{ex}{\rm The construction of $\DL$-cospectral graphs starting with a single graph and adding edges as described in Corollary \ref{cor:cospecAddEdge} is  illustrated in Figure \ref{fig:K1K2}. With $G$, $G_1=G+v_1v_2$, and $G_2=G+v_3v_4$ as shown, $v_1$ and $v_2$ are twins, $v_3$ and $v_4$ are twins,  and $t(v_1)=15=t(v_3)$.  Furthermore, $p_{\DL(G_1)}(x)=p_{\DL(G_2)}(x)=x^8 - 104x^7 + 4601x^6 - 112224x^5 + 1629571x^4 - 14083840x^3 +
67065731x^2 - 135702840x$.
}\end{ex}

\begin{figure}[h!]
\begin{center}
\includegraphics[scale=0.4]{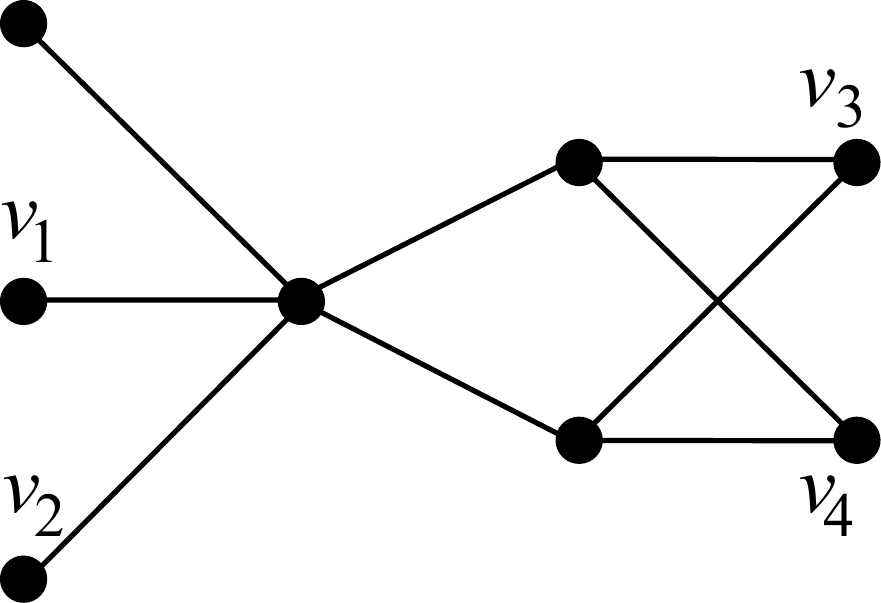}\qquad\quad
\includegraphics[scale=0.4]{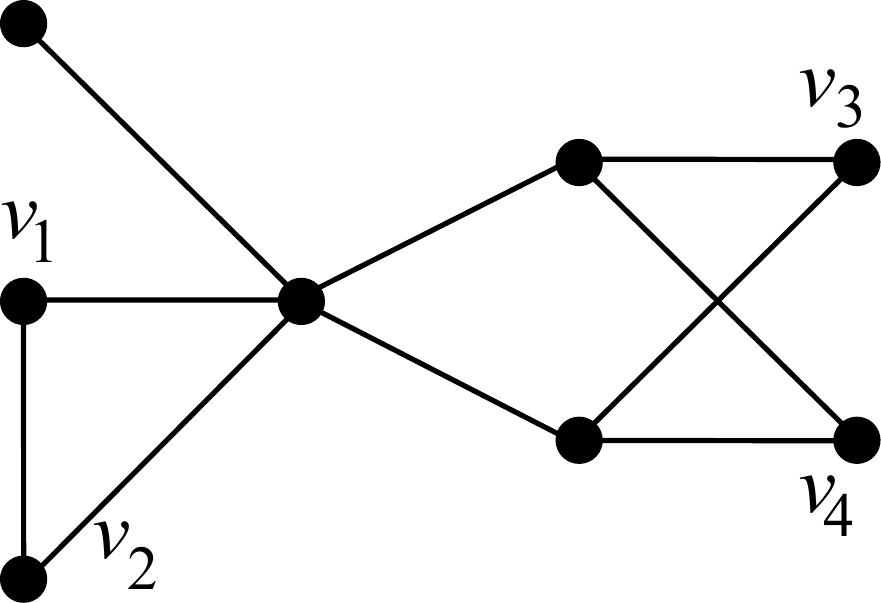}\qquad\quad
\includegraphics[scale=0.4]{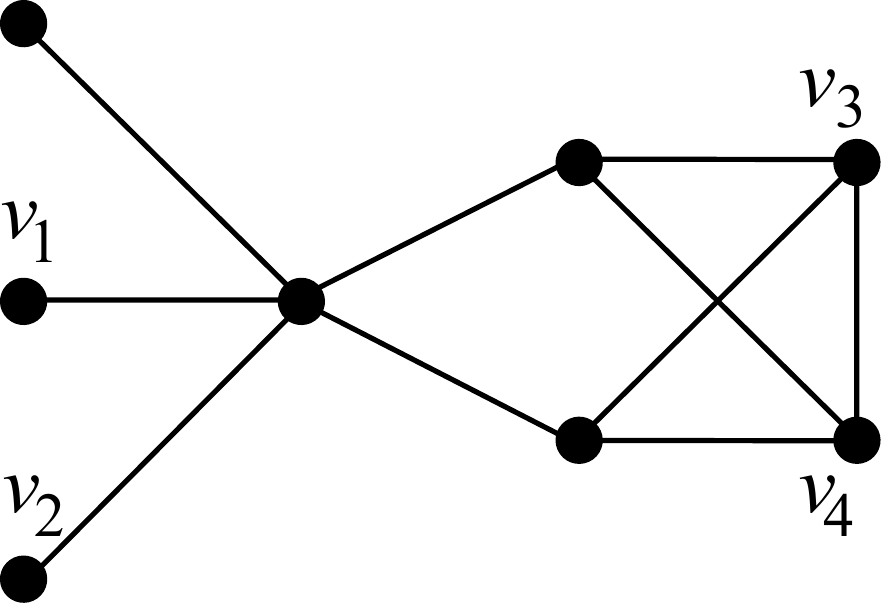}\\
$\null$\qquad$G$\qquad\qquad\qquad\qquad$G_1=G+v_1v_2$\qquad\qquad\qquad$G_2=G+v_3v_4$
\caption{The graphs $G_1=G+v_1v_2$ and $G_2=G+v_3v_4$ produced as in Corollary \ref{cor:cospecAddEdge} \label{fig:K1K2}}\vspace{-15pt}
\end{center}
\end{figure}

Next we show how to create a graph $G$ that has a set of co-transmission twins by {\em twinning} a vertex $v$ in $G$, which adds a new vertex $v'$ to $G$ such that $N(v')=N(v)$. 

\begin{prop}
Let $G$ be a graph with distinct vertices $u$ and $v$ such that $t(u)=t(v)$. Let $G'$ be the graph obtained from $G$ by  twinning  each of  $u$ and $v$ to obtain $u'$ and $v'$, respectively. Then $\{\{u,u'\},\{v,v'\}\}$ is a set of co-transmission twins of $G'$. 
\end{prop}

\begin{proof} Let $d=d_G(u,v)$.  
We know that $u'$ is distance two away from $u$ and $d$ away from $v$, and $v'$ is distance $d$ away from $u$ and two away from $v$. Therefore, we have increased the transmissions of $u$ and $v$ from $G$ to $G'$ by $d+2$. By construction, $t(u')=t(u)$  and $t(v')=t(v)$. Therefore, $\{\{v, v'\}, \{u, u'\}\}$ are co-transmission twins. 
\end{proof}

Next we extend the construction of $\DL$-cospectral graphs from a set of co-transmission twins to sets of vertices called cousins that satisfy somewhat  more relaxed  conditions.

\begin{defn}{\rm 
Let $G$ be a graph of order at least five with $v_1, v_2,v_3, v_4\in V(G)$. Let $C=\{\{v_1,v_2\},$ $\{v_3,v_4\}\}$ and $U(C)=V(G) \setminus \{v_1, v_2, v_3, v_4\}$. Then
$C$ is a  {\em set of cousins} in $G$ if the following conditions are satisfied:
\begin{enumerate}
\item For all $u \in U(C)$, $d_G(u, v_1)=d_G(u, v_2)$ and $d_G(u, v_3)=d_G(u, v_4)$.
\item $\sum_{u \in U(C)} d_G(u, v_1)=\sum_{u \in U(C)} d_G(u, v_3)$.
\end{enumerate}
}\end{defn}

Note that no structure is assumed between $v_1, v_2, v_3, v_4$. In the case when $v_1v_2, v_3v_4\notin E(G)$, it follows directly from the definition that  $N(v_1) \setminus \{v_3, v_4\}=N(v_2) \setminus \{v_3, v_4\}$ and $N(v_3) \setminus \{v_1, v_2\} =N(v_4)\setminus \{v_1,v_2\}$.

\begin{prop}
Let $G$ be a graph with a set of cousins $C=\{\{v_1,v_2\},\{v_3,v_4\}\}$.  Then $N(v_1) \cap N(v_2) \neq \emptyset$ and $N(v_3) \cap N(v_4) \neq \emptyset$.
\end{prop}
\begin{proof}
Since $|V(G)| \geq 5$, $U(C)\ne \emptyset$.  Since $G$ is connected, there must be some vertex  $u\in U(C)$ adjacent to at least one of $v_1,v_2,v_3,v_4$. Without loss of generality, suppose $u\in N(v_1)$. Since $u$ must also be the same distance from $v_2$, $u$ is also adjacent to $v_2$, and thus $N(v_1) \cap N(v_2) \neq \emptyset$. 
Suppose to the contrary that $N(v_3) \cap N(v_4) = \emptyset$, which implies $N(v_3)\cap U(C)=\emptyset$ and $N(v_4)\cap U(C)=\emptyset$. Since $G$ is connected, it follows without loss of generality that $v_3$ is adjacent to $v_1$.  
Then for all $u \in U(C)$, the shortest path between $u$ and $v_3$ uses the edge $v_1v_3$. Thus $d(u, v_3)=d(u, v_1)+1$, and $\sum_{u \in U(C)} d(u,v_3)= (n-4)+\sum_{u \in U(C)} d(u, v_1)$, which is a contradiction. 
\end{proof}

\begin{lem} \label{lem:preserveDistance}
Let $G$ be a graph  with set of cousins $C=\{\{v_1,v_2\},\{v_3,v_4\}\}$ such that $v_1v_2, v_3v_4 \not \in E(G)$. 
Then $d_{G+v_1v_2}(u, v)=d_{G}(u, v)=d_{G+v_3v_4}(u,v)$ for all $u \in U(C)$ 
and $v \in V(G)$. 
\end{lem}

{\begin{proof} 
Let $u \in U(C)$ and $v \in V(G)$.
If a shortest path between $u$ and $v$ in $G+v_1v_2$ does not contain edge $v_1v_2$  and a shortest path between $u$ and $v$ in $G+v_3v_4$ does not contain edge $v_3v_4$,
then the distance between $u$ and $v$ is the same in $G$, $G+v_1v_2$, and $G+v_3v_4$. 
We show that if a shortest path $\PP$ from $u$ to $v$ in $G+v_1v_2$ contains $v_1v_2$, then there is a shorter path   between $u$ and $v$ in $G+v_1v_2$, contradicting the choice of $\PP$.  The argument for  $G+v_3v_4$ is similar.
 
Suppose a shortest  path $\PP$ from $u$ to $v$ in $G+v_1v_2$ contains the edge $v_1v_2$; without loss of generality, $v_1$ precedes $v_2$ in the path order starting from $u$. 
Since $u \in U(C)$, there is a predecessor $w$ of $v_1$ in $\PP$, i.e. $\PP$ contains the subpath $(w,v_1,v_2)$.  
If $w\in U(C)$, then $w \in N(v_1) \cap N(v_2)$ and so replacing $(w,v_1,v_2)$ by $(w,v_2)$ produces a shorter subpath. 
So $w$ is one of $v_3$ or $v_4$, without loss of generality let $w=v_3$. 
Since $v_3\notin U(C)$,  $\PP$ contains the subpath $(u',v_3,v_1,v_2)$ for some $u'\in V(G)$.  Since $v_3v_4\notin E(G+v_1v_2)$, $u'\ne v_4$.  
Thus $u'\in U(C)$.  If $u'v_1\in E(G)$ 
there would be a shorter path between $u$ and $v$ obtained by replacing $(u', v_3, v_1)$ by $(u',v_1)$ 
in $\PP$.  Since this would contradict the choice of $\PP$, $u'v_1\notin E(G)$.  Thus $d_G(u',v_1)=2$, which implies $d_G(u',v_2)=2$ by the definition of a set of cousins.  So there exists $x\in N(u')\cap N(v_2)$.  Then  there is a shorter path between $u$ and $v$ obtained by replacing $(u', v_3, v_1, v_2)$ by $(u',x,v_2)$. 
\end{proof}}

\begin{lem} \label{lem:isomorphism}
Let $H$ be a graph with $V(H) = \{u_1, u_2, u_3, u_4\}$ such that $u_1u_2, u_3u_4 \not \in E(H)$ and suppose  $H + u_1u_2 $ is isomorphic to $ H + u_3u_4$. Then at least one of  the permutations $\sigma_1=(u_1u_4)(u_2u_3)$ or $\sigma_2=(u_1u_3)(u_2u_4)=(u_1u_2)\sigma_1(u_1u_2)$ in the symmetric group on $V(H)$ is an isomorphism of $H+u_1u_2$ and $H+u_3u_4$. 
\end{lem}

\begin{proof}
This lemma was verified by exhaustively checking every labeling of every graph on four vertices using {\em Sage} \cite{sageverify}. \end{proof}


\begin{thm} \label{thm:Cospectral Cousins}
Let $G$ be a graph with a set of cousins $C=\{\{v_1,v_2\},\{v_3,v_4\}\}$ satisfying the following conditions:
\bit
\item Vertices $v_1, v_2$ are not adjacent and $v_3, v_4$ are not adjacent.
\item The subgraph of $G+v_1v_2$ induced by  $\{v_1,v_2,v_3,v_4\}$ is isomorphic to the subgraph of $G+v_3v_4$  induced by $\{v_1,v_2,v_3,v_4\}$.  
\eit
If $G_1=G+v_1v_2$ and $G_2=G+v_3v_4$  are not isomorphic, then they are  $\DL$-cospectral.
\end{thm}

\begin{proof}
By Lemma \ref{lem:preserveDistance}, we know that the entries of $\DL(G_1)$, $\DL(G_2)$, and $\DL(G)$ are equal everywhere except the $4 \times 4$ principal submatrix corresponding to $v_1, v_2, v_3, v_4$. 
Let  $M_1=\DL(G_1)[\{1,2,3,4\}]$ and $M_2=\DL(G_2)[\{1,2,3,4\}]$.  Then we can partition $\DL(G_1)$ and $\DL(G_2)$ as block matrices: \beq\label{eq:mtx-part-3.9}
\DL(G_1)= \left[\begin{array}{c c}
M_1 & Q \\
Q^T & B
\end{array}\right] \text{ and } \DL(G_2)= \left[\begin{array}{cc}
M_2 & Q \\
Q^T & B
\end{array}\right]\eeq where $Q$ is a $4 \x (n-4)$ matrix and $B$ is an $(n-4) \x (n-4)$ matrix. 

From the definition of a set of cousins,  
$\sum_{u \in U(C)} d(u, v_1)=\sum_{u \in U(C)} d(u, v_2)=\sum_{u \in U(C)} d(u, v_3)=\sum_{u \in U(C)} d(u, v_4)$. Thus the row sums of $Q$ are constant. Since every row sum of a distance Laplacian matrix is equal to zero, it follows that the row sums of $M_1$ are constant, the row sums of  $M_2$ are constant, and these constants are equal; denote the constant row sum of $M_i, i=1,2$ by $a$.  
Thus $M_1$ has the form  \[M_1=\left[ \begin{array}{c c c c}
a+1+c+d & -1 & -c & -d \\
-1 & a+1+f+h &-f & -h \\
-c & -f & a+c+f+2 & -2 \\
-d & -h & -2 & a+d+2+h
\end{array}\right]\] for some $c, d, f, h \in \mathbb{Z}$. 
The induced subgraph  $G_1[\{v_1,v_2,v_3,v_4\}]$ is isomorphic to the  induced subgraph of $G_2[\{v_1,v_2,v_3,v_4\}]$ by hypothesis. 
 By Lemma \ref{lem:isomorphism}, we can choose the isomorphism  to be either $\sigma_1$ or $\sigma_2$. So $M_2$ has one  of the forms 
\[M_2^{\sigma_1}=\left[ \begin{array}{cccc}
a+d+h+2 & -2 & -h & -d \\
-2 & a+c+f+2 &-f & -c \\
-h & -f & a+1+f+h & -1 \\
-d & -c & -1 & a+1+c+d
\end{array}\right] \] or \[M_2^{\sigma_2}=
\left[ \begin{array}{cccc}
a+d+h+2 & -2 & -c & -f \\
-2 & a+c+f+2 &-d & -h \\
-c & -d & a+1+f+h & -1 \\
-f & -h & -1 & a+1+c+d
\end{array}\right] \]  corresponding to $\sigma_1$ or $\sigma_2$, respectively.  

Define 
\[S_1= \left[\begin{array}{ c c c c}
 1/2 & 1/2 & 1/2 & -1/2  \\
 1/2 & 1/2 & -1/2 & 1/2  \\
 1/2 & -1/2 & 1/2 & 1/2  \\
 -1/2 & 1/2 & 1/2 & 1/2\end{array}\right]\!, \ \  S_2= \left[\begin{array}{ c c c c}
 1/2 & 1/2 & -1/2 & 1/2  \\
 1/2 & 1/2 & 1/2 & -1/2  \\
 -1/2 & 1/2 & 1/2 & 1/2  \\
 1/2 & -1/2 & 1/2 & 1/2\end{array}\right]\!,\]
and $\mathcal{S}_i=S_i\oplus I_{n-4}$ for $i=1,2$.
Observe that each of $S_1, S_2, \mathcal{S}_1$, and $\mathcal{S}_2$ is its own inverse and transpose ($S_1$ and $S_2$ are scalar multiples of Hadamard matrices). Suppose $\sigma_1$ maps $G_1[\{v_1,v_2,v_3,v_4\}]$ to $G_2[\{v_1,v_2,v_3,v_4\}]$ and 
consider \[\mathcal{S}_1 \DL(G_1) \mathcal{S}_1 = \left[ \begin{array}{c  c}
S_1 M_1S_1 & S_1 Q \\
(S_1 Q)^T & B
\end{array}\right]\!.\]
By direct computation \cite{sageverify}, we can verify that $S_1M_1S_1=M_2^{\sigma_1}$.  Every column of $Q$ is of the form $[-p, -p, -q, -q]^T$, and $S_1[-p, -p, -q, -q]^T=[-p, -p, -q, -q]^T$. Therefore $S_1Q=Q$. 
We have shown that $\mathcal{S}_1 \DL(G_1) \mathcal{S}_1= \DL(G_2)$ and thus the graphs are cospectral. The proof that similarity by $\mathcal{S}_2$ transforms $\DL(G_1)$ to $\DL(G_2)$ when $\sigma_2$ is the isomorphism is analogous.  
\end{proof}

Using a set of co-transmission twins to construct $\DL$-cospectral graphs (as in Corollary \ref{cor:cospecAddEdge}) is a special case of Theorem \ref{thm:Cospectral Cousins}: A straightforward computation shows that  set $\{\{v_1,v_2\},\{v_3,v_4\}\}$ of co-transmission twins is a set of cousins. The twins relationships require that the subgraph $G[\{v_1,v_2,v_3,v_4\}]$ is four isolated vertices or the complete bipartite graph with parts $\{v_1,v_2\}$ and $\{v_3,v_4\}$, which implies the subgraph of $(G+v_1v_2)[\{v_1,v_2,v_3,v_4\}]$ is isomorphic to  $(G+v_3v_4)[\{v_1,v_2,v_3,v_4\}]$. 
 
Our next main result uses a set of cousins $C=\{\{v_1, v_2\}, \{v_3, v_4\}\}$ but toggles edges $v_1v_3$ and $v_2v_4$ rather than $v_1v_2$ and $v_3v_4$.   We need a preliminary lemma.

\begin{lem}\label{lem:preserveDistancePlus} 
Let $G$ be a graph with a set of cousins $C=\{\{v_1, v_2\}, \{v_3, v_4\}\}$ such that $v_1v_3, v_2v_4 \not \in E(G)$,  for every  $x \in N(v_1) \cap N(v_2)$ there exists $y\in N(v_3) \cap N(v_4)$ such that $xy \in E(G)$, and for every  $y \in N(v_3) \cap N(v_4)$ there exists $x\in N(v_1) \cap N(v_2)$ such that $xy \in E(G)$. 
Then $d_{G+v_1v_3}(u, v)=d_{G}(u, v)=d_{G+v_2v_4}(u, v)$ for all $u \in U(C)$ and $v \in V(G)$. 
\end{lem}

\begin{proof}
Let $u \in U(C)$ and $v \in V(G)$.
If a shortest path between $u$ and $v$ in $G+v_1v_3$ does not contain edge $v_1v_3$  and a shortest path between $u$ and $v$ in $G+v_2v_4$ does not contain edge $v_2v_4$,
then the distance between $u$ and $v$ is the same in $G$, $G+v_1v_3$, and $G+v_2v_4$.
We show that there is a shortest path that does not include $v_1v_3$ between $u$ and $v$ in $G+v_1v_3$; analogously, there is a shortest path in $G+v_2v_4$ that does not contain edge $v_2v_4$.

Suppose a shortest path $\mathcal{P}$ from $u$ to $v$ in $G+v_1v_3$ contains the edge $v_1v_3$; without loss of generality, $v_1$ precedes $v_3$ in the path order starting from $u$.  Since $u\in U(C)$, there is a predecessor $x$ of $v_1$ in $\mathcal{P}$, i.e. $\mathcal{P}$ contains the subpath $(x,v_1, v_3)$. We first show that  $x\in U(C)$.  Since $\mathcal{P}$ is a path, $x\ne v_1,v_3$.  Since $(x,v_1,v_3)$ is a subpath of a shortest path between $u$ and $v_3$, it cannot be replaced by a shorter subpath.  For $u'\in U(C)$, a subpath of the form 
$(u',v_2,\dots,v_1,v_3)$ can be replaced by $(u',v_1,v_3)$, and a subpath of the form 
$(u',v_4,\dots,v_1,v_3)$ can be replaced by $(u',v_3)$.  Thus $x\ne v_2,v_4$, so $x\in U(C)$.
Because $x\in U(C)$, there exists $y\in N(v_3) \cap N(v_4)$ such that $xy \in E(G)$,  the subpath $(x,v_1,v_3)$ can be replaced by $(x,y,v_3)$, creating a path of equal length that avoids edge $v_1v_3$.
\end{proof} 

\begin{thm} \label{thm:Cospectral Cousins v1v3}
Let $G$ be a graph with a set of cousins $C=\{\{v_1, v_2\}, \{v_3, v_4\}\}$ satisfying the following conditions:\vspace{-5pt}
\bit
\item Vertices $v_1, v_3$ are not adjacent and $v_2, v_4$ are not adjacent.
\item The subgraph of $G+v_1v_3$ induced by $\{v_1,v_2,v_3,v_4\}$  is isomorphic to the subgraph of $G+v_2v_4$ induced by $\{v_1,v_2,v_3,v_4\}$  via the permutation $\sigma_1=(14)(23)$.
\item For every  $x \in N(v_1) \cap N(v_2)$ there exists $y\in N(v_3) \cap N(v_4)$ such that $xy \in E(G)$, and for every  $y \in N(v_3) \cap N(v_4)$ there exists $x\in N(v_1) \cap N(v_2)$ such that $xy \in E(G)$.
\eit
If $G_1=G+v_1v_3$ and $G_2=G+v_2v_4$  are not isomorphic, then they are  $\DL$-cospectral.
\end{thm}
\begin{proof}
By Lemma \ref{lem:preserveDistancePlus},  the entries of $\DL(G_1)$ are equal to $\DL(G_2)$ everywhere except the $4 \times 4$ principal submatrix corresponding to $v_1,v_2,v_3,v_4$. The $4 \times 4$ principal submatrices have  structure analogous to $M_1$ and to $M_2$ with $\sigma_1$, and the remainder of the proof follows the same method. 
 \end{proof}

The reason we assume that the isomorphism of $G_1[\{v_1,v_2,v_3,v_4\}]$ and $G_2[\{v_1,v_2,v_3,v_4\}]$ is  $\sigma_1=(14)(23)$ is because isomorphic graphs are not cospectral by definition. Applying Lemma \ref{lem:isomorphism} with $v_2=u_3$ and $v_3=u_2$ produces $\sigma_1=(v_1 \ v_4)(v_2 \ v_3)$ and $\sigma_2=(v_1 \ v_2)(v_3 \ v_4)$.  Suppose $G_1[\{v_1,v_2,v_3,v_4\}]$  is isomorphic to $G_2[\{v_1,v_2,v_3,v_4\}]$ via $\sigma_2$.  Let $P_{\sigma_2}$ denote the $4\x4$ permutation matrix representing $\sigma_2=(12)(34)$ and $P_2=P_{\sigma_2}\oplus I_{n-4}$.  With the partition of $\DL(G_1)$ as in \eqref{eq:mtx-part-3.9}, $P_{\sigma_2}Q=Q$.  Thus $P_2\DL(G_1)P_2=\DL(G_2)$ and $G_1$ is isomorphic to $G_2$.

We now discuss an example of Theorem \ref{thm:Cospectral Cousins v1v3} on a family of cospectral bipartite graphs. This family contains the only bipartite pair of cospectral graphs on eight vertices.

\begin{defn}
For $k \in \mathbb{Z}_+$,  $B_k$ is the  graph defined as follows (and shown in Figure \ref{fig:bipartiteFamilyGen} not including the dashed edges):  Five vertices are denoted by $v_0, v_1, v_2, v_3, $ and $v_4$, and the set of edges of the induced subgraph $B_k[v_0, v_1, v_2, v_3,v_4]$ is $\{v_0v_1, v_0v_2, v_2v_3\}$.   There are $k$ twin leaf vertices adjacent to $v_0$;   the set of these vertices is denoted by $L$. There are $k+1$ twin  vertices of degree three each  adjacent   to  $v_0, v_3,$ and $v_4$;  the set of these vertices is denoted by $R$. 
\end{defn}

\begin{figure}[h!]
\begin{center}
\includegraphics[scale=0.4]{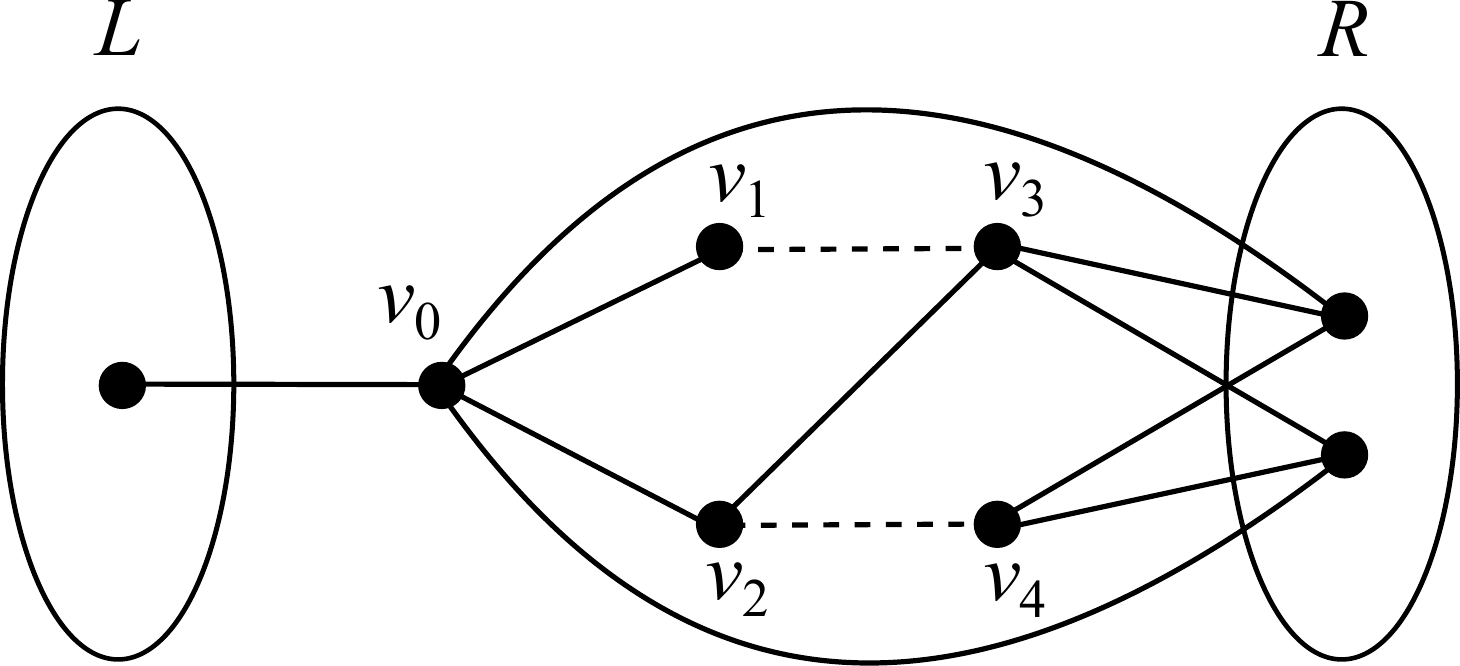}
\caption{A bipartite family $B_k$ where  the dashed edges are toggled to create $\DL$-cospectral graphs. The set $L$ contains $k$ twin vertices and the set $R$ contains $(k+1)$ twin vertices.}
\label{fig:bipartiteFamilyGen}
\end{center}
\end{figure}

\begin{thm}\label{thm:bipart} The graphs $B_k+v_1v_3$ and $B_k+v_2v_4$ are $\DL$-cospectral, and both are bipartite.
\end{thm}

\begin{proof}
By construction,  $B_k$ is a bipartite graph parts $\{v_0, v_3, v_4\}$ and $V(B_k)\setminus \{v_0, v_3, v_4\}$. The edges $v_1v_3$ or $v_2v_4$ each have endpoints in different parts, so adding these edges keeps the graph bipartite.

We show that $B_k$ satisfies Theorem \ref{thm:Cospectral Cousins v1v3}. First we show that $\{\{v_1, v_2\}, \{v_3, v_4\}\}$ are a set of cousins. 
By inspection, we see that for all $u \in L$, $d(u, v_1)=2=d(u, v_2)$ and $d(u, v_3)=3=d(u, v_4)$. For all $u \in R$,  $d(u, v_1)=2=d(u, v_2)$ and $d(u, v_3)=1=d(u, v_4)$. Additionally $d(v_0, v_1)=1=d(v_0, v_2)$ and $d(v_0, v_3)=2=d(v_0, v_4)$. 
Finally,
\[\sum_{u \in U(C)} d(u, v_1) = \sum_{u \in L}d(u, v_1) + d(v_0, v_1)+\sum_{u \in R} d(u, v_1)=2k+1+2(k+1) =4k+3.\]
\[\sum_{u \in U(C)} d(u, v_3)= \sum_{u \in L}d(u, v_3) + d(v_0, v_3)+\sum_{u \in R} d(u, v_3)=3k+2+(k+1)=4k+3.\]
Therefore,  $C=\{\{v_1, v_2\}, \{v_3, v_4\}\}$ is a set of cousins. 
In $B_k$, we have that $v_1$ is not adjacent to $v_3$ and $v_2$ is not adjacent to $v_4$. The subgraph of $B_k+v_1v_3$ induced by $\{v_1, v_2, v_3, v_4\}$ is isomorphic to the subgraph of $B_k+v_2v_4$ induced by $\{v_1, v_2, v_3, v_4\}$. Additionally, the only vertex in $N(v_1) \cap N(v_2)=\{v_0\}$  is adjacent to every vertex in $R=N(v_3) \cap N(v_4)$.  Thus,  $B_k +v_1v_3$ and $B_k+v_2v_4$ are $\DL$-cospectral by Theorem \ref{thm:Cospectral Cousins v1v3}.  
\end{proof} 

The family $B_k$ grows by increasing $k$, or twining two vertices (one in $L$ and one in $R$) at the same time to preserve the cousin property. This growth is not unique to this cospectral family: We can change the set of edges between the vertices $\{v_1,v_2,v_3,v_4\}$ provided   the cousin property and the induced graph isomorphism are preserved.  
Furthermore, $L$ and $R$ can have arbitrary structure since it does not affect the cousin property nor the induced graph isomorphism; however adding edges within $L$ or $R$ will result in the graphs that are not bipartite. 

\section{Cospectral 
transmission regular graphs}\label{s:TR}

If two graphs are transmission regular and $\mathcal{D}$-cospec-tral, then they are $\mathcal{D}^L$-cospectral (see Remark \ref{rem:D-DL} for more detail). Thus, results for $\D$-cospectrality can be combined with characterizations of transmission regular graphs in order to obtain results about $\DL$-cospectrality.
In this section we  discuss transmission regular $\DL$-cospectral graphs, including strongly regular graphs, distance-regular graphs, and circulant graphs. 
We begin our discussion of transmission regular graphs by generalizing a well-known example of a transmission regular graph that is not  regular
(the graph  in Figure \ref{fig_H_3}) to an  infinite family of graphs with this property.

\begin{figure}[h]
\begin{center}
\includegraphics[scale=0.3]{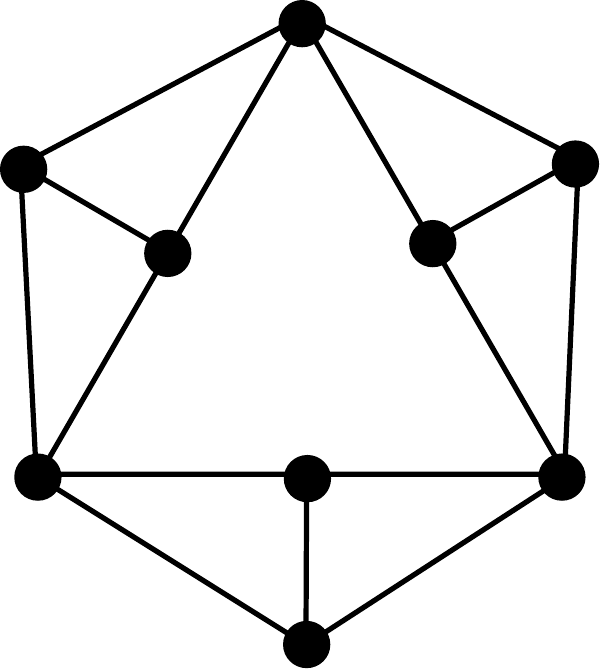} 
\end{center}
\caption{The graph $H_3$, which is transmission regular but not regular.}
\label{fig_H_3}
\end{figure}

\begin{defn}{\rm 
Let $H_n$ be the graph with vertex set $V=\{v_1,\ldots,v_{2n},u_1,\ldots,u_n\}$ and edge set $E=\{v_iv_{i+1}:1\leq i\leq 2n-1\}\cup\{v_{2n}v_1\}\cup \{u_iv_{2i-1},u_iv_{2i},u_iv_{2i+1}:1\leq i\leq n-1\}\cup \{u_nv_{2n-1},u_nv_{2n},u_nv_1\}$.
}\end{defn}

\begin{prop} For each  $n\geq 2$, $H_n$ is not regular.
For each odd $n\geq 3$, $H_n$ is   transmission regular and $t(H_n)=\frac{3n^2+1}{2}$.
\end{prop}
\bpf
Since $H_n$ has vertices of degrees three and four,  $H_n$ is not regular.  
By the symmetry of $H_n$, all vertices  $u_i$ and $v_{2i}$ in the definition of $H_n$ have the same transmission for $1\leq i\leq n$; likewise, all vertices $v_{2i-1}$ have the same transmission as $v_1$. Thus, to show that $H_n$ is transmission regular, it suffices to show that $t(u_1)=t(v_1)$. Using the identities $1+3+\dots+k=\left(\frac{k+1}{2}\right)^2$ and $2+4+\dots+k=\frac{k}{2}\left(\frac{k}{2}+1\right)$, we have
\begin{eqnarray*}
t(u_1)&=&3\, (1)  +  4\,(2)+ 2\,(3)+ 4\,(4)+ 2\,(5)+\dots+2\,(n-2)+4\,(n-1)+1\,(n)\\
&=&4\left(\frac{n-1}{2}\left(\frac{n-1}{2}+1\right)\right)+2\left(\left(\frac{n-2+1}{2}\right)^2\right)+1+n\\
&=&\frac{3n^2+1}{2}\\
&=&4\left(\frac{n-2+1}{2} \right)^2+2\left(\frac{n-1}{2}\left(\frac{n-1}{2}+1\right)\right)+2n\\
&=&4\,(1) + 2\,(2)+ 4\,(3)+2\,(4)+4\,(5)+\dots+4\,(n-2)+2\,(n-1)+2\,(n)\\
&=&t(v_1),
\end{eqnarray*}
where the first and last equalities can be easily verified by inspection by counting the distances from $u_1$ and $v_1$ to all other vertices. 
\epf

\begin{rem}\label{rem:D-DL}{\rm  Suppose $G$ is a transmission regular graph of order $n$.  The spectral radius of the distance matrix  $\D(G)$ is the transmission of $G$ and the trace of the distance Laplacian is $\tr(\DL(G))=nt(G)$.   Observe that  $\dLspec(G)=\{\dlev_1=0,\dlev_2,\dots,\dlev_n\}$ where $\dlev_k=t(G)- \dev_k$ for $k=1,\dots,n$  (recall that distance eigenvalues are ordered largest to smallest, whereas distance Laplacian eigenvalues are ordered smallest to largest).  
Thus, two transmission regular graphs  are $\D$-cospectral if and only if they are $\DL$-cospectral.  
}\end{rem}

\subsection{Strongly regular and distance-regular graphs}\label{ssSRG}

Strongly regular, and more generally distance-regular graphs, provide important examples of transmission regular $\DL$-cospectral graphs.  A $k$-regular graph $G$ of order $n$ is {\em strongly regular} with parameters
$(n, k, \LL,\mu)$ if every pair of adjacent vertices has $\LL$ common neighbors
and every pair of nonadjacent distinct vertices has $\mu$ common neighbors.   A connected strongly regular graph $G$ has 
$\diam(G)\le  2$.
There are well known formulas for computing the adjacency eigenvalues of $G$ and their multiplicities from the strongly regular graph parameters  (see \cite[Section 10.2]{GR01}).  Formulas for the distance eigenvalues and their multiplicities have also been determined \cite[p. 262]{AP15} and \cite{GRWC15}.  Thus the distance Laplacian eigenvalues can be readily determined, as done in the next remark. 

\begin{rem}\label{rem:DL-SRG} Let $G$ be a strongly regular graph with parameters $(n,k,\LL,\mu)$.  Then any vertex $v$ has $k$ neighbors and  $n-k-1$ vertices at distance two,  so $G$ is $(2n-k-2)$-transmission regular.    
Thus the distance Laplacian eigenvalues of $G$ are 
    \bea0 && \mbox{ of multiplicity } 1\\
2n-k+  \frac{1}{2}\left(\lambda-\mu+\sqrt{(\lambda-\mu)^2+4(k-\mu)}\right)&& \mbox{ of multiplicity }\frac{1}{2}
\left(n-1-\frac{2k+(n-1)(\lambda-\mu)}{\sqrt{(\lambda-\mu)^2+4(k-\mu)}}\right)\\
2n-k+\frac{1}{2}\left(\lambda-\mu-\sqrt{(\lambda-\mu)^2+4(k-\mu)}\right)&& \mbox{ of multiplicity }\frac{1}{2}
\left(n-1+\frac{2k+(n-1)(\lambda-\mu)}{\sqrt{(\lambda-\mu)^2+4(k-\mu)}}\right).
\eea
Furthermore, two strongly regular graphs with the same set of parameters are $\DL$-cospectral. 
\end{rem}

There are examples, such as those presented next,  of non-isomorphic strongly regular graphs with the same parameters $(n,k,\LL,\mu)$; such graphs are therefore  $\DL$-cospectral and transmission regular (so neither strong regularity nor transmission regularity implies a graph is spectrally determined).  The notation $\lambda^{(m)}$ means that the multiplicity of eigenvalue $\lambda$ is $m$.

\begin{ex}\label{ex:S-H24}{\rm  The Shrikhande graph, shown in Figure \ref{fig:Shrik}, is the graph $S=(V,E)$ where \\$V=\{0, 1, 2, 3\} \x\{0, 1, 2, 3\}$ and $E=\{\{(a, b),(c, d)\}:  (a, b) -(c, d) \in\{\pm(0, 1), \pm(1, 0), \pm(1, 1)\}\}$. Both $S$ and the 2,4-Hamming graph $H(2,4):=K_4\cp K_4$ are strongly regular with parameters $(16,6,2,2)$.  They are 24-transmission regular with  distance spectrum  ${\dspec(S)}=\left\{24, 0^{(9)}, (-4)^{(6)}\right\}$ and distance Laplacian spectrum ${\dLspec(S)}=\left\{0, 24^{(9)}, 28^{(6)}\right\}$.  Since $S$ and $H(2,4)$ are cospectral and strongly regular, so are their complements $\OL S$ and $\OL{H(2,4)}$.  Observe that $\omega(S)=3=\alpha(\OL{S})$ and $\omega(H(2,4))=4=\alpha(\OL{H(2,4)})$.}
\begin{figure}[h!]
\begin{center}
\scalebox{.25}{\includegraphics{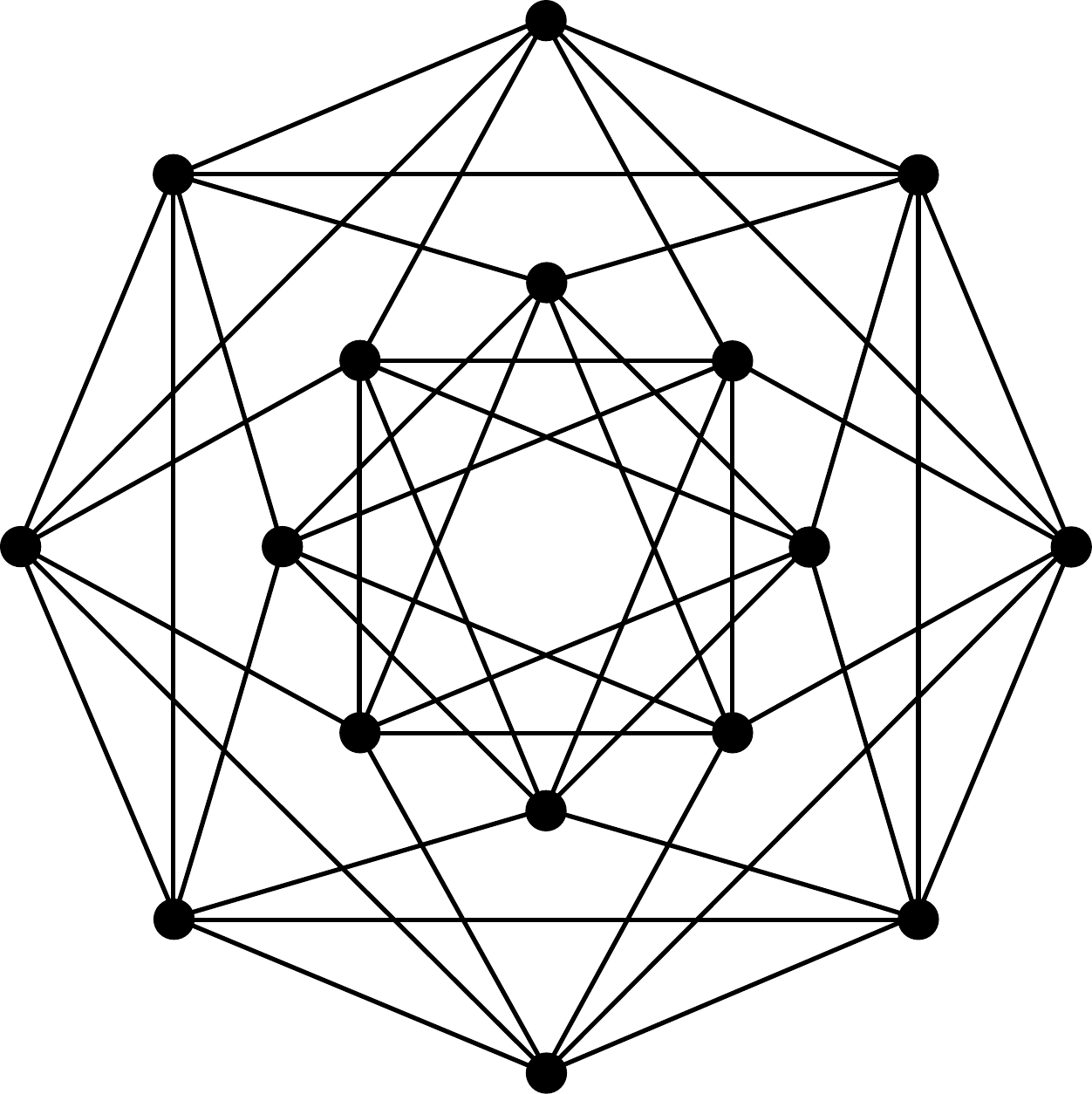}}
\caption{The Shrikhande graph $S$\label{fig:Shrik}}\vspace{-10pt}
\end{center}
\end{figure}
\end{ex}

\begin{ex}\label{ex:DH}{\rm   Let  $i, j, k\ge 0$ be  integers. The graph $G =(V, E)$ is  {\em distance-regular} if for any choice of $u, v\in V$ with $d(u, v) =k$, the number of vertices $w\in V$ such that $d(u, w) =i$ and $ d(v, w) =j$ is independent of the choice of $u$ and $v$. Distance spectra of several families of distance-regular graphs were determined in \cite{AP15} and the distance spectra of  all distance-regular graphs having exactly one positive distance eigenvalue are listed in \cite{GRWC15}.  A {\em Doob graph $D(m,n)$} of diameter $2m+n$ is the Cartesian product of $m$ copies of the Shrikhande graph with one copy of the Hamming graph $H(n,4)$.  It is known \cite{GRWC15} that  the Doob graph $D(m,n)$ and the Cartesian product of $m$ copies of $H(2,4)$ with $H(n,4)$ are $\D$-cospectral and thus $\DL$-cospectral. 
}\end{ex}

\begin{ex}{\rm  The triangular graph $T_8$ (which is the  line graph of $K_8$) and the three Chang graphs are all strongly regular with the parameters $(28, 12, 6, 4)$  (see \cite{KS1994} for definitions of the Chang graphs).  Thus  all four have the same distance spectrum  $\{42, 0^{(20)}, (-6)^{(7)}\}$ and the same $\DL$-spectrum $\{0, 42^{(20)}, 48^{(7)}\}$.}
\end{ex}


\subsection{Cospectral circulant graphs}\label{sscirc}

 The ring of integers mod $n$ is denoted by $\Z_n$, with elements $0,\dots,n-1$, and we view an element of $\Z_n$ as also having a meaningful interpretation as an integer, where we can use the standard ordering of elements, although arithmetic is performed modulo $n$.  For an integer $n\ge 3$ and a set $S\subseteq \Z_n$ that does not contain zero, the {\em circulant graph} defined by $n$ and $S$  is the graph  with vertex set $\Z_n$ and edges of the form $ij$ where $i-j$ or $j-i$ is in $S$; this graph is denoted by $\Circ n S$. The set $S$ is called the {\em connection set} of $\Circ n S$. By convention we assume that   $1\le x\le \lf\frac n 2\rf$ for every $x\in S$.  If $q$ is relatively prime to $n$ then $\Circ n S$ is isomorphic to $\Circ n {qS}$ where $qS=\{qx:x\in S\}$.  Thus it is common to require that $1\in S$ if $S$ contains any element relatively prime to $n$, and we make this assumption also.  Note that $\Circ n S$ is connected if and only if the greatest common divisor of the elements of $S$ together with $n$ is one.  A circulant graph is both regular and transmission regular. 
Circulants are an interesting family of graphs that play an important role in some applications of graph theory such as computing and quantum spin networks \cite{BPS, Ilic}. They also provide examples of distance Laplacian (and distance) cospectral mates.
 
   Table \ref{tab:circ} lists the orders and connection sets for all sets of $\DL$-cospectral circulant graphs  of order at most 20; the data in the table were produced by computations in {\em Sage} \cite{sagecode}. The smallest order is 16, and  the $\DL$-cospectral  graphs $\circs {16} {1,2,8}$ and $\circs {16} {1,6,8}$ are shown in Figure \ref{fig:circ}.

\begin{table}[h!]
\begin{center}
 \begin{tabular}{|c|c|}\hline
$n$  & $S$ \\ 
\hline
16 & \{1, 2, 8\}, \{1, 6, 8\}\\ \hline
18 & \{1, 2, 9\}, \{1, 4, 9\}\\ \hline
18 & \{1, 6, 8\}, \{2, 3, 6\}\\ \hline
20 & \{1, 2, 3, 8\}, \{1, 2, 7, 8\} \\ \hline
20 & \{1, 2, 4, 9\}, \{1, 6, 8, 9\}\\ \hline
20 & \{1, 2, 4, 10\}, \{1, 6, 8, 10\}\\ \hline
20 & \{1, 2, 4, 9, 10\}, \{1, 6, 8, 9, 10\}\\ \hline
20 & \{1, 2, 3, 8, 10\}, \{1, 2, 7, 8, 10\} \\ \hline
20 & \{1, 2, 4, 5, 9\}, \{1, 5, 6, 8, 9\}\\ \hline
20 & \{1, 2, 3, 5, 8\}, \{1, 2, 5, 7, 8\} \\ \hline
20 & \{1, 2, 4, 5, 9, 10\}, \{1, 5, 6, 8, 9, 10\}\\ \hline
20 & \{1, 2, 3, 5, 8, 10\}, \{1, 2, 5, 7, 8, 10\}\\ \hline
\end{tabular}
 \end{center} 
\caption{Order $n$ and connection sets $S$ for small circulant $\DL$-cospectral pairs
}\label{tab:circ}\vspace{-10pt}
 \end{table}

\begin{figure}[!h]\begin{center}
\scalebox{.25}{\includegraphics{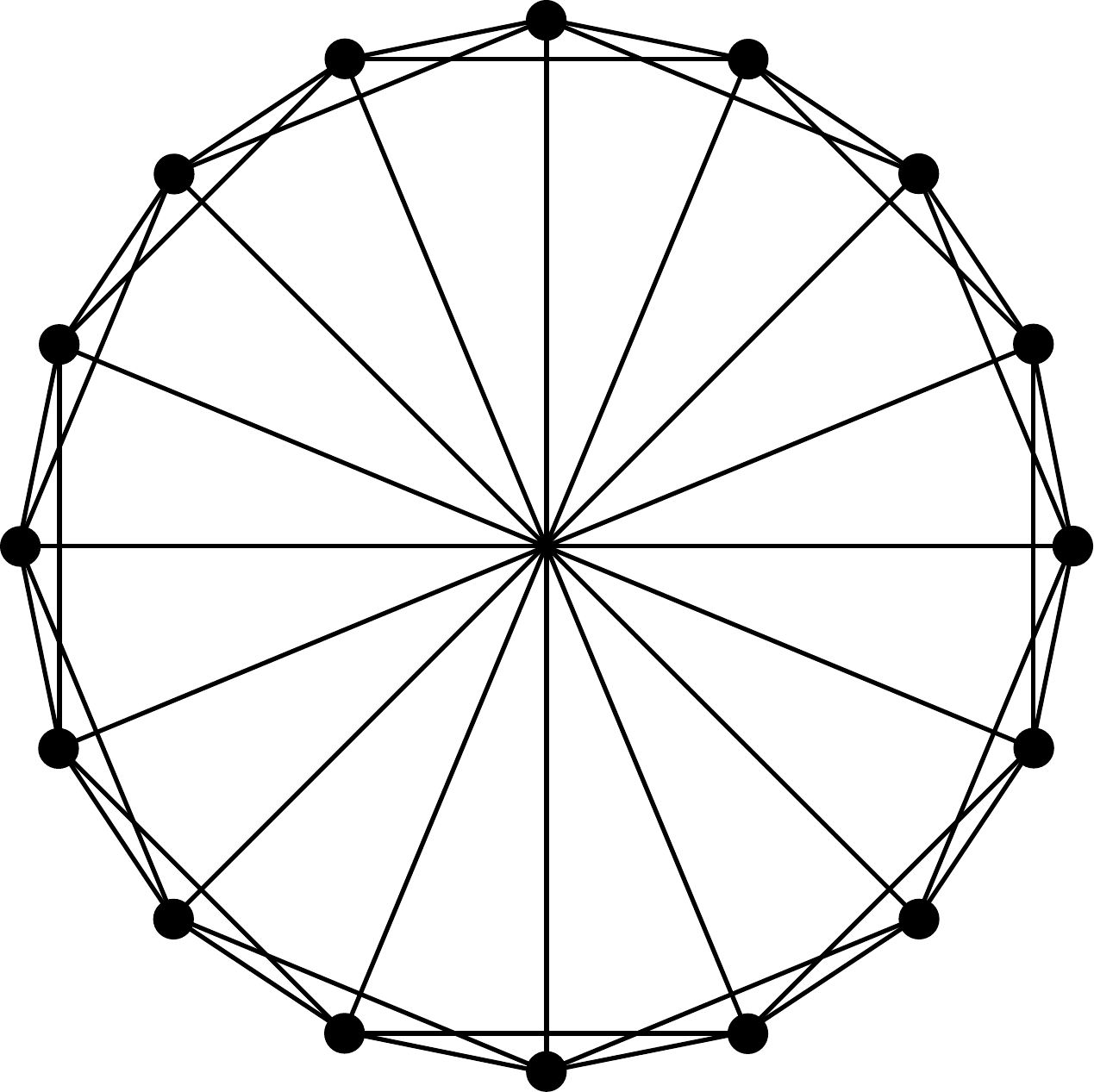}\qquad\includegraphics{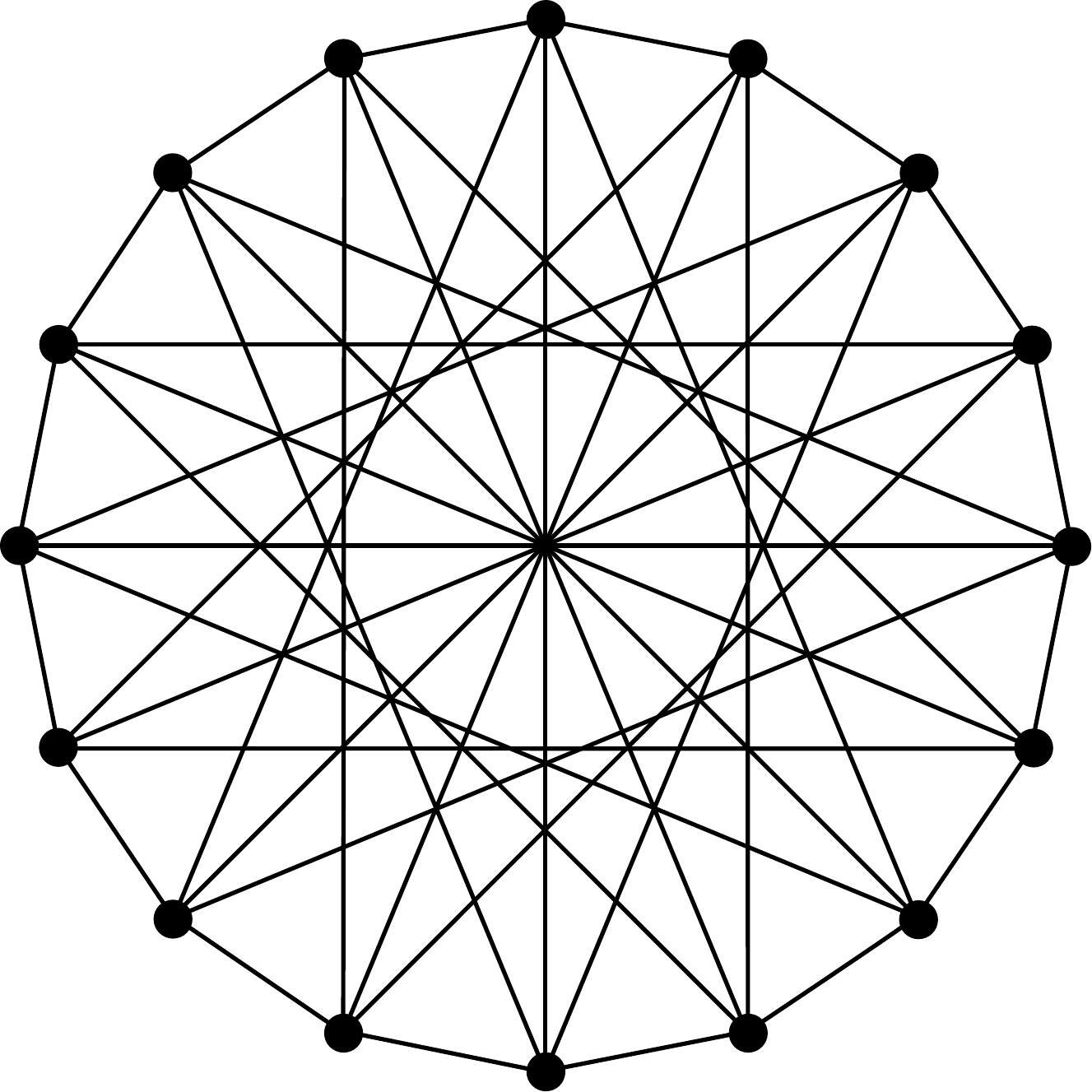}}
\caption{The circulant graphs $\circs {16} {1,2,8}$ and $\circs {16} {1,6,8}$ \label{fig:circ}}\vspace{-10pt}
\end{center}
\end{figure}

Although Table \ref{tab:circ} contains only cospectral pairs of even order, by using  {\em Sage} to perform a search we found that the graphs $\circs {27} {1,2,9}$ and $\circs {27} {1,4,9}$ on 27 vertices are $\DL$-cospectral.
 
    The {\em Sage} search that produced Table \ref{tab:circ} was restricted to circulant graphs, so it would not detect a pair of $\DL$-cospectral graphs where one is a circulant and the other is not.   
 Recall a result from \cite{BM,HvW,S.Ma}  that links strongly regular graphs and circulant graphs: If a connected circulant graph $\Circ n S$  is strongly regular, then $n$ must be a prime $p$ satisfying $p\equiv 1\mod 4$ and $\Circ p S$ must be isomorphic to a Paley graph on the field of order $p$ (see \cite{GR01} for the definition of a Paley graph). 

\begin{ex}\label{Paley29} 
{\rm The graph  $\circs{29}{1, 4, 5, 6, 7, 9, 13}$,  which is shown in Figure \ref{fig:circpaley}, is  a Paley graph on $29$ vertices that is strongly regular with  parameters $(29,14,6,7)$.  There are 40 additional nonisomorphic  strongly regular graphs on 29 vertices with the same strongly regular graph parameters \cite{spence-web} that are necessarily $\DL$-cospectral.  The $\DL$-spectrum of $\circs{29}{1, 4, 5, 6, 7, 9, 13}$ is $\left\{0^{(1)},\frac{87-\sqrt{29}}{2}^{(14)}, \frac{87+\sqrt{29}}{2}^{(14)}\right\}$; the $\DL$-spectrum can be computed from the formulas for the $\DL$-spectrum of a strongly regular graph (see Remark  \ref{rem:DL-SRG} and  \cite{sageverify}).  }
\end{ex}
\begin{figure}[h!]
\begin{center}
\includegraphics[scale=0.25]{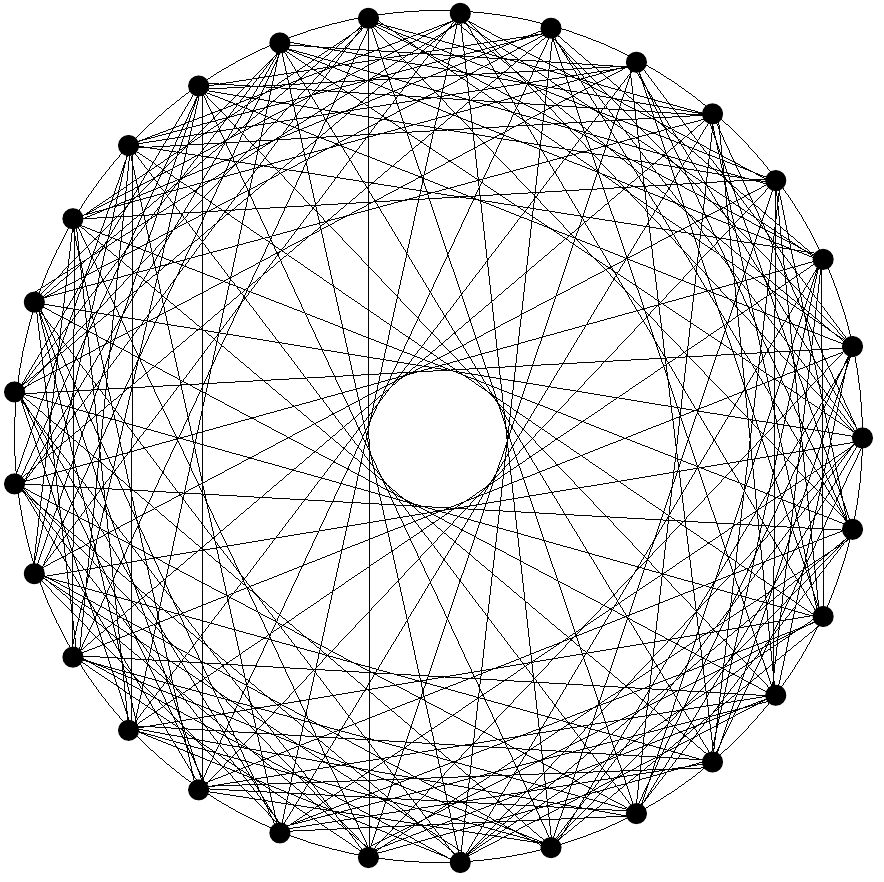}
\caption{The circulant graph $\circs{29}{1, 4, 5, 6, 7, 9, 13}$}
\label{fig:circpaley}
\end{center}
\end{figure}

A {\em circulant matrix} is a square matrix in which each row is formed from the row above by a right cyclic shift of plus one. A trivial example of a circulant matrix is that of the identity matrix $I_n$ where this shift pattern can be seen along the lead diagonal. 
 The adjacency, distance, and distance Laplacian matrices  of a circulant graph are circulant matrices.  
 
 For $j=1,\dots, n-1$, define $d_j^{(r)}$ as the distance between vertex $0$ and vertex $j$ in $\ccirc n r$. Thus, $d_j^{(r)}=\left\lceil\frac{j}{r}\right\rceil$ for $1\leq j\leq \left\lfloor\frac{n}{2}\right\rfloor$ and due to the symmetry of the circulant graphs $d^{(r)}_{n-j}=d^{(r)}_j$.
Define  $d_0^{(r)}=t(\ccirc n r)=\sum^{n-1}_{j=1}d_j^{(r)}$.  Then the first row of  the circulant matrix $\DL(\ccirc n r)$ is given by
\[
\left(d_0^{(r)},-d^{(r)}_1,\dots,-d^{(r)}_{n-1}\right)
\]
The remaining entries of the matrix $\DL(\ccirc n r)$ are defined by using the circulant matrix property. 


\begin{prop}\label{trans lem} The transmission of any consecutive circulant graph $\ccirc n r$ is 
\[t(\ccirc n r)=\left(\left\lfloor\frac{n-1}{2r}\right\rfloor+1\right)\left((n-1)-r\left\lfloor\frac{n-1}{2r}\right\rfloor\right).
\]
\end{prop}

\bpf For any vertex $v$ of $\ccirc n r$, there are $2r$ vertices at distance one from $v$, $2r$ vertices at distance two, and so on, ending with $2r$ vertices at distance $\lf\frac{n-1}{2r}\rf$.  The remaining $n-1-2r\lf\frac{n-1}{2r}\rf$ vertices  are at distance $\lf\frac{n-1}{2r}\rf+1$.  Thus the transmission is
\bea
t(G)=t(v)&=&2r\left(\sum^{\left\lfloor\frac{n-1}{2r}\right\rfloor}_{i=0}i\right)+\left((n-1)-2r\left\lfloor\frac{n-1}{2r}\right\rfloor\right)\left(\left\lfloor\frac{n-1}{2r}\right\rfloor+1\right)\\
&=& \left(\left\lfloor\frac{n-1}{2r}\right\rfloor+1\right)\left((n-1)-r\left\lfloor\frac{n-1}{2r}\right\rfloor\right). \qquad\qquad \cvd
\eea
\epf

We can use Proposition \ref{trans lem} to easily calculate the Weiner index from the transmission.
\begin{cor}{ The Weiner index of a consecutive circulant is
\[W\left(\ccirc n r\right)=\frac{n}{2}t(G)
=\frac{n}{2}\left(\left\lfloor\frac{n-1}{2r}\right\rfloor+1\right)\left((n-1)-r\left\lfloor\frac{n-1}{2r}\right\rfloor\right).
\]
}\end{cor}
 Since the pair $(n,r)$ determines both the number of eigenvalues and the transmission (by Proposition \ref{trans lem}), and the transmission equals the distance spectral radius, the next result is immediate.

\begin{cor}  For $r_j\le \lf \frac {n_j} 2\rf, j=1,2$, the consecutive circulants $\ccirc{n_1}{r_1}$ and $\ccirc{n_2}{r_2}$ are distance cospectral (and hence distance Laplacian cospectral) if and only if $n_1=n_2$ and $r_1=r_2$.
\end{cor}


\section{Unimodality of the coefficients of the distance Laplacian characteristic polynomial}\label{s:unimod}

In their study of distance spectra of trees, Graham and Lov\'asz \cite{GL78} showed that for a tree $T$ and $\delta_k(T)$ denoting the  coefficient of $x^k$ in $\det(\D(T)-xI)=(-1)^np_{\D(G)}(x)$, the quantity $d_k(T)=(-1)^{n-1}\delta_k(T)/2^{n-k-2}$ is determined as a fixed linear combination of the numbers of certain subtrees in $T$; the values $d_k(T)$ are called the {\em normalized coefficients} of $T$.  
A sequence $a_0,a_1,\ldots, a_n$ of real numbers is {\em unimodal} if there is a $k$ such that $a_{i-1}\leq a_{i}$ for $i\leq k$ and $a_{i}\geq a_{i+1}$ for $i\geq k$.  Graham and Lov\'asz \cite{GL78}  conjectured that  the sequence  $d_0(T),\dots,d_{n-2}(T)$ of normalized coefficients is unimodal with the maximum value occurring at $\lf \frac n 2\rf$  for  a tree $T$ of order $n$.  Collins \cite{C89} showed that the location of the maximum for a path on $n$ vertices is approximately $n\left(1-\frac 1 {\sqrt 5}\right)$ and with modern computers it is straightforward to verify that this  location of the peak is strictly greater than    $\lf \frac n 2\rf$ for  $n\ge 18$ vertices (and for odd $n\ge 9$).  Aalipour et al.~\cite{GRWC15-uni} established the unimodality of the normalized coefficients and also of the absolute values of the coefficients of the distance characteristic polynomial.  These ideas naturally suggest exploring the behavior of  the coefficients of distance Laplacian characteristic polynomial.

The  unimodality of the sequence of absolute values of the coefficients of the distance characteristic polynomial is equivalent to that of the normalized coefficients \cite{GRWC15-uni, GL78}, and the proof used log-concavity and its connection to polynomials with real roots.  A sequence $a_1,a_2,\ldots, a_n$ of real numbers
 is {\em log-concave} if $a_j^2\geq a_{j-1}a_{j+1}$ for all $j=2,\dots,  n-1$.  A polynomial $p$ is \emph{real-rooted} if all roots of $p$ are real (by convention, constant polynomials are considered real-rooted).  The next  result describes the connections between real-rootedness, log-concavity, and unimodality, and was the key to the proof.
 
 \begin{lem}\label{unimodequiv} {\rm \cite{GRWC15-uni, brandensurvey}}
\begin{enumerate}[(a)]
\item\label{rrlc} If $p(x)=a_nx^n+\dots + a_1x+a_0$ is a real-rooted polynomial, then the coefficient sequence $a_i$ of $p(x)$ is log-concave.

\item\label{lcum} If a sequence $a_0,a_1,a_2,\dots, a_n$ is positive and log-concave, then $a_0,a_1,a_2,\ldots, a_n$ 
is unimodal.
\end{enumerate}
\end{lem}

\begin{cor}\label{imply-unimodal} Let $A$ be a real symmetric matrix.  The  coefficient sequence $a_0,a_1,\dots, a_{n-1},a_n=1$ of $p_A(x)$ is log-concave.  If  a subsequence $a_s,\dots, a_t$ of $a_0,a_1,\dots, a_{n-1},a_n$ is positive then this subsequence is also unimodal.  If all entries of a subsequence of $a_0,a_1,\dots,a_{n-1}, a_n$  alternate in sign or all are negative, then the subsequence of absolute values $|a_s|,  \dots,  |a_t|$ is unimodal.
\end{cor} 
\bpf  Since all the eigenvalues of $A$ are real, the log-concavity of  $a_0,\dots, a_{n-1},a_n$ follows immediately from Lemma \ref{unimodequiv}, and any positive subsequence is unimodal.  If the signs of the subsequence $a_s,\dots, a_t$ alternate, or if all values are negative, then $|a_j^2|=a_j^2\geq a_{j-1}a_{j+1}=|a_{j-1}| |a_{j+1}|$ for $ s+1\le j \le t=1$, so the positive subsequence $|a_s|,  \dots,   |a_t|$ is also log-concave, and therefore unimodal. 
\epf

 The {\em $\ell$th symmetric function of $\alpha_1,\dots,\alpha_n$}, denoted by $S_\ell(\alpha_1,\dots,\alpha_n)$ is the sum over all products of $\ell$ distinct elements of $\{\alpha_1,\dots,\alpha_n\}$.  Observe that if $0=\alpha_1=\dots=\alpha_c<\alpha_{c+1}<\dots<\alpha_n$, then $S_\ell(\alpha_1,\dots,\alpha_n)=S_\ell(\alpha_{c+1},\dots,\alpha_n)$ for $\ell\le n-c$ and $S_\ell(\alpha_1,\dots,\alpha_n)=0$ for $\ell>n-c$.
 
\begin{thm} Let $G$ be a connected graph, and let $p_{\DL(G)}(x)=x^n+\delta^L_{n-1}x^{n-1}+\dots+\delta^L_{1}x$.  Then the sequence $\delta^L_{1},\dots,\delta^L_{n-1},\delta^L_{n}=1$ is log-concave and $|\delta^L_{1}|,\dots,|\delta^L_{n}|$ is unimodal.  In fact, $|\delta^L_{1}|\geq \dots \geq |\delta^L_{n}|$. 
\end{thm} 

{\em Proof.} It is well known that the distance Laplacian is positive semidefinite (this follows immediately from Gershgorin's disk theorem), and that zero is a simple eigenvalue (see, for example, \cite{AH13}).  Since $\delta^L_{k}=$\break$(-1)^{n-k} S_{n-k}(\dlev_1,\dlev_2,\dots,\dlev_n)=(-1)^{n-k} S_{n-k}(\dlev_2,\dots,\dlev_n)$ and $\dlev_2,\dots,\dlev_n$  are positive, all the nonzero coefficients alternate in sign. By Corollary 
\ref{imply-unimodal} the sequence $\{|\delta_k^L|\}_{k=1}^n$ is unimodal.

To show that $|\delta^L_{1}|\geq \dots \geq |\delta^L_{n}|$, by unimodality it is sufficient to show that $|\delta^L_1| > |\delta^L_2|$. 
Since  $\partial^L_2 \geq n$ \cite[Corollary 3.6]{AH13}, 
\[ |\delta^L_2| = \sum\limits_{i=2}^n \prod\limits_{j \neq i} \partial^L_j \leq (n-1) \prod\limits_{j=3}^{n} \partial^L_j < \prod\limits_{j=2}^n \partial^L_j = |\delta^L_1|. \cvd \]


\section{Concluding remarks}

In this paper, we studied graph cospectrality with respect to the distance Laplacian matrix.  We showed that several global properties of a graph, such as degree sequence, transmission sequence, diameter, and number of nontrivial automorphisms, and also several local properties,   such as the presence of a leaf, a dominating vertex, and a cut vertex, are not preserved by $\DL$-cospectrality.
It would be interesting to determine whether or not other properties,  such as those listed in the next question, 
are preserved by $\DL$-cospectrality. 
\begin{quest}\label{q:preserve} Are the following  properties  preserved by $\DL$-cospectrality?
\ben
\item Acyclicity (being a tree).
\item Bipartiteness.  
\item\label{tr} Transmission regularity. 
 \item Regularity.
 \item Strong regularity. 
\een
\end{quest}

 It is not known whether trees are $\DL$-spectrally determined \cite{AH13}; if so, then the property of being a tree would necessarily be preserved by $\DL$-cospectrality.
For each of the other properties in Question \ref{q:preserve}, examples of a $\DL$-cospectral pair with both graphs having the property and a $\DL$-cospectral pair with neither graph having the property have been presented  (in Sections \ref{ss:preserve} and \ref{s:TR}). However, we do not have examples of $\DL$-cospectral pairs where one member has the property and the other does not.

We also established several general methods for producing $\DL$-cospectral graphs based on sets of vertices sharing common neighborhoods, transmissions, and partial transmissions. We used these constructions to produce infinite families of $\DL$-cospectral graphs. Future work could include  finding additional methods for constructing $\DL$-cospectral graphs. Conversely, it would also be interesting to find certain local structures or general properties of a graph that guarantee the graph has no $\DL$-cospectral mates; one conjectured such property is acyclicity. 

Since two transmission regular graphs are $\D$-cospectral if and only if they are $\DL$-cospectral, results about $\D$-cospectrality of transmission regular graphs lead to results about $\DL$-cospectrality. To this end, we studied various families of strongly regular, distance-regular, and circulant graphs, and showed that many of them have $\DL$-cospectral mates. We also constructed an infinite family of graphs which are transmission regular but not  regular; it is an open question to determine whether this family of graphs is $\DL$-spectrally determined. 

Finally, we established that the absolute values of coefficients of the distance Laplacian characteristic polynomial are unimodal, and in fact decreasing. It would be interesting to establish other structural results about the coefficients of the distance Laplacian characteristic polynomial and their relation to $\DL$-cospectrality.


\section*{Acknowledgements}  
 We gratefully acknowledge financial support for this research from the following grants and organizations: NSF grants \#1603823 and \#1604458,  NSA grant \#H98230-18-1-0017, and the Combinatorics Foundation. \vspace{-3pt}



\begin{thebibliography}{99}\vspace{-6pt}

\bibitem{GRWC15} G. Aalipour, A. Abiad, Z. Berikkyzy, J. Cummings, J. De Silva, W. Gao, K. Heysse, L. Hogben, F.H.J. Kenter, J.C.-H. Lin, M. Tait.  On the distance spectra of graphs. {\it{Linear Algebra Appl.}} 497 (2016), 66--87.

\bibitem{GRWC15-uni} G. Aalipour, A. Abiad, Z. Berikkyzy, L. Hogben, F.H.J. Kenter, J.C.-H. Lin, M. Tait.  Proof of a conjecture of Graham and Lovasz concerning unimodality of coefficients of the distance characteristic polynomial of a tree.  {\em Electron. J. Linear Algebra} 34 (2018), 373--380. 






\bibitem{AH13}  M. Aouchiche, P. Hansen. Two Laplacians for the distance matrix of a graph. {\em Linear Algebra Appl.} 439 (2013), 21--33. 
 
\bibitem{AH14} M. Aouchiche, P. Hansen. Distance spectra of graphs: A survey. {\em Linear Algebra Appl.} {458} (2014),
301--386.


\bibitem{AP15} F. Atik, P. Panigrahi. On the distance spectrum of distance regular graphs. {\em Linear Algebra Appl.} {478} (2015), 256--273.

\bibitem{BPS} M. Basi\'{c}, M. Petkovi\'{c}, D. Stevanovi\'{c}. Perfect state transfer in integral circulant graphs. \textit{Applied. Math. Let.} 22 (2009),  1117--1121.

\bibitem{brandensurvey} P. Br{\"{a}}nd{\'{e}}n. Unimodality, log-concavity, real-rootedness, and
  beyond. In \emph{Handbook of Enumerative Combinatorics}, M. Bona, Editor,  CRC Press, Boca Raton, 2015. 
  
  
\bibitem{BM} W.G. Bridge, R.A. Mena. Rational circulant with rational spectra and cyclic strongly regular graphs. {\em Ars Combin.} 8 (1979), 143--161, 







\bibitem{C89} K.L. Collins. On a conjecture of Graham and Lov\'asz about distance matrices. {\em Discrete Appl. Math.} {25} (1989), 27--35.

\bibitem{sageverify}	 K. Duna, L. Hogben, K. Lorenzen.
{\em Sage} verifications of examples and computational lemmas.  PDF available at \url{https://orion.math.iastate.edu/lhogben/BDGLRSY18-Distance_Laplacian--Sage.pdf}.  {\em Sage} worksheet available at  \url{https://sage.math.iastate.edu/home/pub/90/}.  

 \bibitem{sagecode}	 K. Duna, K. Lorenzen.
{\em Sage} code for finding $\DL$-copsectral graphs. PDF available at \url{https://orion.math.iastate.edu/lhogben/Sage_Code_BDHLRSY18--Sage.pdf}. {\em Sage} worksheet available at  \url{https://sage.math.iastate.edu/home/pub/91/}.  

\bibitem{sagedata} K. Duna, K. Lorenzen.  List of all $\DL$-cospectral graphs of order at most 10 in {\em graph6} format.  PDF available at \url{https://orion.math.iastate.edu/lhogben/Data_orders78910_BDHLRSY18.pdf}.

\bibitem{GM82} C.D. Godsil, B.D. McKay. Constructing cospectral graphs. {\em Aequationes Math.}
25 (1982), 257--268.

\bibitem{GR01}
C. Godsil, G. Royle. {\it Algebraic Graph Theory.} Springer-Verlag, New York, 2001.


\bibitem{GL78}  R.L. Graham, L. Lov\'asz. Distance matrix polynomials of trees.  {\em Adv. Math.} {29} (1978), 60--88.

\bibitem{GP71}  R.L. Graham, H.O. Pollak. On the addressing problem for loop switching. {\em Bell Syst. Tech. J.} {50}
(1971), 2495--2519.

 
\bibitem{H17} K. Heysse. A construction of distance cospectral graphs. {\em Linear Algebra Appl.} 535 (2017), 195--212.

\bibitem{HvW} D.R. Hughes, J.V. van Lint, R.M. Wilson. Announcement at the Seventh British Combinatorial Conference, Cambridge, 1979 (unpublished).

\bibitem{Ilic} A. Ili\'{c}. Distance spectra and distance energy of integral circulant graphs. 
\textit{Linear Algebra Application.} 433 (2010), 1005--1014.


 
\bibitem{KS1994} J.H. Koolen, S.V. Shpectorov. Distance-regular Graphs the Distance Matrix of which has Only One Positive Eigenvalue. \emph{European J. Combin.} {15} (1994), 269--275.


\bibitem{S.Ma} S.L. Ma. Partial difference sets. {\em Discrete Math.} 52 (1984), 75--89.

\bibitem{NP} M. Nath, S. Paul. On the distance Laplacian spectra of graphs. {\em Linear Algebra Appl.} 460 (2014), 97-110.

\bibitem{spence-web} E. Spence.  Strongly Regular Graphs on at most 64 vertices. Available at \url{http://www.maths.gla.ac.uk/~es/srgraphs.php}.




\end{thebibliography}
\end{document}